\pdfoutput=1
\documentclass{amsart}

\usepackage{graphicx,mathrsfs,amssymb,amsmath}

\usepackage[accsupp]{axessibility}

\usepackage[pdfdisplaydoctitle=true,colorlinks=true,urlcolor=blue,citecolor=blue,linkcolor=blue,
pdfstartview=FitH,pdfpagemode=None,bookmarksnumbered=true]{hyperref}
\usepackage{cmap} 

\newtheorem{theorem}{Theorem}[section]
\newtheorem{lemma}[theorem]{Lemma}
\newtheorem{corollary}[theorem]{Corollary}
\newtheorem{proposition}[theorem]{Proposition}
\newtheorem{remark}[theorem]{Remark}
\newtheorem{definition}[theorem]{Definition}
\newtheorem{example}[theorem]{Example}

\numberwithin{equation}{section}

\newcommand{\cz}{{\mathbb C}}
\newcommand{\gz}{{\mathbb Z}}
\newcommand{\nz}{{\mathbb N}}
\newcommand{\rz}{{\mathbb R}}
\newcommand{\sz}{{\mathbb S}}

\newcommand{\bfA}{\mathbf{A}}
\newcommand{\bflambda}{\boldsymbol{\lambda}}

\newcommand{\bfS}{\mathbf{S}}
\newcommand{\wtbfL}{\mathbf{\widetilde{L}}}
\newcommand{\wtbfS}{\mathbf{\widetilde{S}}}

\newcommand{\calA}{\mathcal{A}}
\newcommand{\calK}{\mathcal{K}}

\newcommand{\scrC}{\mathscr{C}}
\newcommand{\scrL}{\mathscr{L}}
\newcommand{\scrS}{\mathscr{S}}

\newcommand{\dbar}{d\hspace*{-0.08em}\bar{}\hspace*{0.1em}}
\newcommand{\eps}{\varepsilon}
\newcommand{\forget}[1]{}
\newcommand{\lra}{\longrightarrow}
\newcommand{\rpbar}{\overline{\rz}_+}
\newcommand{\spk}[1]{\langle#1\rangle}
\newcommand{\trinorm}[1]{|\hspace*{-1pt}|\hspace*{-1pt}|#1|\hspace*{-1pt}|\hspace*{-1pt}|}
\newcommand{\wh}{\widehat}
\newcommand{\whsz}{\widehat{\mathbb S}}
\newcommand{\wt}{\widetilde}

\begin{document}
\title[$\psi$do with point-singularity in the covariable]{
Parametric pseudodifferential operators\\ with point-singularity in the covariable}

\author{J\"org Seiler}
\address{Dipartimento di Matematica, Universit\`{a} di Torino, Italy}
\email{joerg.seiler@unito.it}

\begin{abstract}
Starting out from a new description of a class of parameter-de\-pen\-dent 
pseudodifferential operators  with finite regularity number due to G. Grubb, 
we introduce a calculus of parameter-dependent, poly-homogeneous symbols whose homogeneous 
components have a particular type of point-singularity in the covariable-parameter space. 
Such symbols admit intrinsically a second kind of expansion which is closely related to 
the expansion in the Grubb-Seeley calculus and permits to 
recover the resolvent-trace expansion for elliptic pseudodifferential oerators originally 
proved by Grubb-Seeley. Another application is the invertibility of parameter-dependent operators 
of Toeplitz type, i.e., operators acting in subspaces determined by zero-order 
pseudodifferential idempotents. 

\vspace*{5mm}

\noindent
\textbf{Keywords:} Pseudodifferential operators; parameter-ellipticity; resolvent; trace expansion; operators of Toeplitz type
\end{abstract}

\vspace*{-20mm}

\maketitle

\tableofcontents

\section{Introduction} 

In the present paper we develop a calculus of parameter-dependent pseudodifferential 
operators ($\psi$do), both for operators in Euclidean space $\rz^n$ and for operators 
on sections of vector-bundles over closed Riemannian manifolds, which is closely related 
to Grubb's calculus of operators with finite regularity number \cite{Grub} and to the 
Grubb-Seeley calculus introcuced in \cite{GrSe}. The calculus allows to obtain the classical 
resolvent-trace expansion for elliptic $\psi$do due to \cite{GrSe} and a systematic treatment 
of  $\psi$do of Toeplitz type in the sense of \cite{Seil12,Seil15}. 

At the base of our calculus lies a ``geometric'' characterization of the above mentioned 
regularity number: consider a parameter-dependent $\psi$do $a(D;\mu)$ with symbol 
$a(\xi;\mu)$ depending, for simplicity, only on the covariable $\xi\in\rz^n$ and the 
parameter $\mu\in\rpbar$. The symbol $a$ belongs to the parameter-dependent 
\emph{poly-homogeneous H\"ormander class} $S^d$ if it admits an asymptotic expansion 
\begin{equation}\label{eq:introA}
 a(\xi;\mu)\sim\sum_{j=0}^{+\infty} a_j(\xi;\mu)
\end{equation}
with symbols $a_j\in S^{d-j}_{hom}$ that are positively homogeneous in $(\xi,\mu)$ of degree $d-j$. 
If $\sz^{n}_+=\{(\xi,\mu)\mid |\xi|^2+\mu^2=1\}$ denotes the unit semi-sphere, then 
\begin{equation}\label{eq:introB}
 a_j(\xi;\mu)=|(\xi,\mu)|^{d-j}\;\wh{a}_j\Big(\frac{(\xi,\mu)}{|(\xi,\mu)|}\Big),\qquad 
 \wh{a}_j\in\scrC^\infty(\sz^{n}_+).
\end{equation}
The operator is parameter-elliptic if the homogeneous principal symbol $a_0$ never vanishes, 
i.e., $\wh{a}_0$ does never vanish on the unit semi-sphere. In this case there exists a 
parametrix $b(\xi;\mu)\in S^{-d}$ such that $b(D;\mu)$ is the inverse of $a(D;\mu)$ 
for large $\mu$. 

In Section \ref{sec:04} we show that $a(D;\mu)$ is an operator of order $d$ and with 
\emph{regularity number} $\nu\in\rz$ in the sense of \cite{Grub} if $a$ admits a 
decomposition $a=\wt{a}+p$ with $p\in S^d$ and where $\wt{a}$ admits an expansion 
of the form \eqref{eq:introA}, with components satisfying \eqref{eq:introB} but with 
\emph{singular} functions $\wh{a}_j$: introducing polar coordinates $(r,\phi)$ on 
$\sz^{n}_+$, centered in the ``north-pole'' $(\xi,\mu)=(0,1)$, they belong to the 
\emph{weighted} space $r^{\nu-j}\scrC^\infty_B(\whsz^{n}_+)$, where 
$\whsz^n_+=\sz^{n}_+\setminus\{(0,1)\}$ and $\scrC^\infty_B$ means smooth functions 
which remain bounded on $\whsz^n_+$ after arbitrary applications of totally characteristic 
derivatives $r\partial_r$ and usual derivatives in $\phi$. 

This observation leads us to consider symbols $a=\wt{a}+p$ with $p\in S^d$ but where the homogeneous components of $\wt{a}$ originate from the weighted spaces $r^{\nu-j} \scrC^\infty_T(\whsz^{n}_+)$, $\nu\in\gz$, where $\scrC^\infty_T(\whsz^{n}_+)$ is the space of all functions on $\whsz^{n}_+$ that, in coordinates $(r,\phi)$, extend smoothly up to and including $r=0$ $($the subscript $T$ stands for Taylor expansion$)$. Symbols of this kind do not only have an expansion \eqref{eq:introA} but intrinsically a further expansion of the form 
\begin{align}\label{eq:intro01}
 {a}(\xi;\mu)\sim\sum_{j=0}^{+\infty}a_{[\nu+j]}^{\infty}(\xi)\,[\xi,\mu]^{d-\nu-j}, \qquad
 a_{[\nu+j]}^{\infty}(\xi)\in S^{\nu+j}(\rz^n),
\end{align}
where $[\xi,\mu]$ denotes a smooth function that coincides with the usual modulus away from the origin and $S^{m}(\rz^n)$ is the standard poly-homogeneous H\"ormander class of order $m$ without parameter. See Section \ref{sec:05} for details. Evidently, the expansion \eqref{eq:intro01} resembles the one employed by Grubb-Seeley in \cite{GrSe}. While Grubb-Seeley's expansion is in powers of $\mu$ and has its origin in a meromorphic (at infinity) dependence on the parameter $\mu$, \eqref{eq:intro01} directly originates from the Taylor expansion of the homogeneous components and makes no use of a holomorphic dependence on the parameter.  However, expanding $[\xi,\mu]^m$ in powers of $\mu$ allows us to obtain a Grubb-Seeley expansion and ultimately we can recover the resolvent-trace expansion of $\psi$do shown in \cite{GrSe}. This is discussed in detail in Sections \ref{sec:06} and \ref{sec:07.6}. 

Ellipticity in our class is most simple for a positive regularity number $\nu>0$. In this case, the homogeneous principal symbol extends by continuity to the north-pole, and its non-vanishing yields the existence of a parametrix which is the inverse of $a(D;\mu)$ for large values of the parameter $\mu$. For $\nu=0$, ellipticity is more involved and two additional symbolic levels come into play: 
\begin{itemize}
\item[(a)] the \emph{principal angular symbol} which originates from the leading term of the Taylor expansion of the homogeneous principal symbol, 
\item[(b)] the \emph{principal limit-symbol}, i.e., the symbol $a^\infty_{[0]}$ from \eqref{eq:intro01}.
\end{itemize}
Non-vanishing of the homogeneous principal symbol, of the principal angular symbol, and invertibility of the \emph{operator} $a^\infty_{[0]}(D)$ guarantee the existence of a parametrix in the class which is the inverse for large values of $\mu$. Concerning $(a)$, our calculus appears to be related with Savin-Sternin \cite{SaSt} where a similar structure occurs. 

We show that our calculus of operators on $\rz^n$ is invariant under changes of coordinates, see Section \ref{sec:07.1}. Thus we can define corresponding classes of $\psi$do on closed manifolds $M$, acting on sections of finite-dimensional vector-bundles. While the homogeneous principal symbol and the principal angular symbol have a global meaning as bundle morphisms on $(T^*M\times\rpbar)\setminus0$ and $T^*M\setminus0$, respectively, the expansion \eqref{eq:intro01} is shown to have a global analog too, namely 
\begin{align}\label{eq:intro02}
 A(\mu)\sim\sum_{j=0}^{+\infty}A_{[\nu+j]}^{\infty}\,\Lambda^{d-\nu-j}(\mu), \qquad
 A_{[\nu+j]}^{\infty}\in L^{\nu+j}(M;E_0,E_1),
\end{align}
where $\Lambda^\alpha(\mu)\in L^\alpha(E_0,E_0)$, $\alpha\in\rz$, denote elliptic elements 
in H\"ormander's class with (scalar) homogeneous principal symbol $(|\xi|^2_x+\mu^2)^{\alpha/2}$, where $|\cdot|$ refers to some fixed Riemannian metric on $M$. This is shown in Section \ref{sec:07.2}. The so-called \emph{limit-operator} $A_{[\nu]}^\infty$ takes the place of the above used limit-symbol. In Section \ref{sec:07.7} we discuss an application to parameter-dependent $\psi$do of Toeplitz type, here on closed manifolds; originally the concept of Toeplitz type operators emerged in the study of boundary value problems with Atiyah-Patodi-Singer type boundary conditions, see \cite{Schu37,SSS98}. 

In the present paper we limit ourselves to $\psi$do on $\rz^n$ or closed manifolds. However, it 
is a natural question whether the established calculus allows to build up a corresponding 
calculus for boundary value problems, in the spirit of \cite{Grub,Grub01} and 
\cite{ScSc1}, leading to a parameter-dependent version of the classical 
Boutet de Monvel algebra \cite{Bout1}. 
Similarly, one could address the analogous question for manifolds with singularities 
(conical singularities, in the simplest case), following and extending the approach of Schulze 
\cite{Schu91,Schulze-Wiley}. We plan to address these questions in future work. 

Hoping to help the reader in reading this paper, we finish this introduction by listing the most important  
spaces of pseudodifferential symbols used in the sequel:
\begin{align*}
S^d_{1,0}(\rz^n),S^d(\rz^n): & & & \text{Section \ref{sec:02.2}} \\ 
S^{d,\nu}_{1,0},S^{d,\nu},S^{d,\nu}_{hom}: & & & 
 \text{Definitions \ref{def:Grubb01}, \ref{def:poly}, and \ref{def:hom-grubb}} \\
S^{d}_{1,0},S^{d},S^{d}_{hom}: & & &  
 \text{Definitions \ref{def:Hoer01}, \ref{def:strongly-hom}, and \ref{def:Hoer-poly}}\\ 
\wt{S}^{d,\nu}_{1,0},\wt{S}^{d,\nu},\wt{S}^{d,\nu}_{hom}: & & & 
 \text{Definitions \ref{def:weak01}, \ref{def:weak03}, and \ref{def:weak02}}\\ 
\wtbfS^{d,\nu}_{1,0},\wtbfS^{d,\nu},\wtbfS^{d,\nu}_{hom}: & & &  
 \text{Definitions \ref{def:symbol-expansion}, \ref{def:wtbfs-poly} and \ref{def:angular}}\\ 
\mathbf{S}^{d,\nu},\mathbf{S}^{d,\nu}_{hom}: & & & 
 \text{Definition \ref{def:new_reg}}
\end{align*}

\section{Notations, symbols, and Leibniz product} 
\label{sec:02}

\subsection{Basic notations} 
\label{sec:02.1}

Let $\spk{y}=(1+|y|^2)^{1/2}$ for $y\in\rz^m$ with arbitrary $m$. 
Let $y\mapsto[y]:\rz^m\to\rz$ denote a smooth, strictly positive function that coincides 
with the modulus $|y|$ outside the unit-ball. If $y=(\xi,\mu)$, we write shortly 
$|\xi,\mu|:=|(\xi,\mu)|$, $\spk{\xi,\mu}=\spk{(\xi,\mu)}$, and $[\xi,\mu]:=[(\xi,\mu)]$. 

A \emph{zero-excision function} on $\rz^m$ is a smooth function $\chi(y)$ that vanishes in a 
neighborhood of the origin and such that $(1-\chi)(y)$ has compact support. 

If $f,g:\Omega\to\rz$ are two functions on some set $\Omega$ we shall write $f\lesssim g$ or 
$f(\omega)\lesssim g(\omega)$ if there exists a constant $C\ge 0$ such that 
$f(\omega)\le Cg(\omega)$ for every $\omega\in\Omega$.  

Let $f(\omega,y)$ be defined on a set of the form $\Omega\times(\rz^m\setminus\{0\})$. 
With slight abuse of language, we shall call $f$ \emph{homogeneous} of degree $d$ in $y$ if 
$$f(\omega,ty)=t^d f(\omega,y)\qquad \forall\;(\omega,y)\quad\forall\; t>0;$$
it would be correct to use the terminology \textit{positively} homogeneous, but for brevity we shall 
not do so. Suppose $y=(u,v)$ with $u\in\rz^k$ and $v\in\rz^{m-k}$ $(k$ may be equal to $m$, i.e.,
$y=u)$. We shall say that $f$ is homogeneous of degree $d$ in $(u,v)$ for large $u$ if there exists 
a constant $R\ge 0$ (frequently assumed to be equal to $1)$ such that
$$f(\omega,tu,tv)=t^d f(\omega,u,v)\qquad \forall\;(\omega,v)\quad 
    \forall\;|u|\ge R\quad\forall\; t\ge1.$$

\subsection{H\"ormander's class} 
\label{sec:02.2}

The uniform H\"or\-mander class $S^d_{1,0}(\rz^n)$ consists of those symbols  
$a(x,\xi):\rz^n\times\rz^n\to\cz$ satisfying the uniform estimates
  $$|D^\alpha_\xi D^\beta_x a(x,\xi)|\lesssim \spk{\xi}^{d-|\alpha|}$$
for every multi-indices $\alpha,\beta\in\nz_0^n$. This is a Fr\'{e}chet space with the system of norms 
\begin{equation}\label{eq:semi-norm01}
 \|a\|_j=\max_{|\alpha|+|\beta|\le j}
 \sup_{x,\xi}|D^\alpha_\xi D^\beta_x a(x,\xi)|\spk{\xi}^{|\alpha|-d},\qquad j\in\nz_0.
\end{equation}
A symbol $a(x,\xi)\in S^d_{1,0}(\rz^n)$ is called \emph{poly-homogeneous} 
if there exist smooth functions $a_{\ell}(x,\xi)$ defined on $\rz^n\times(\rz^n\setminus\{0\})$ 
which are homogeneous of degree $d-\ell$ in $\xi$ and satisfy 
  $$|D^\alpha_\xi D^\beta_x a_{\ell}(x,\xi)|\lesssim |\xi|^{d-\ell-|\alpha|},$$
such that 
 $$r_{a,N}(x,\xi):=a(x,\xi)-\chi(\xi)\sum_{\ell=0}^{N-1}a_{\ell}(x,\xi)\in 
   S^{d-N}_{1,0}(\rz^n)$$
for every $N$, where $\chi(\xi)$ is an arbitrary fixed zero-excision function 
$($note that $r_{a,0}=a)$. 
Denote by $S^d(\rz^n)$ the space of all such symbols. 
It is a Fr\'{e}chet space with the system of norms $\|a\|_{j,N}:=\|r_{a,N}\|_j$, $j,N\in\nz_0$, and 
\begin{equation}\label{eq:semi-norm02}
\|a\|_{j,\ell}^\prime=\max_{|\alpha|+|\beta|\le j}\sup_{x,|\xi|=1}
|D^\alpha_\xi D^\beta_x a_{\ell}(x,\xi)|,\qquad j,\ell\in\nz_0.
\end{equation}
 
The $\psi$do associated with $a(x,\xi)$, denoted by $a(x,D)$, is   
 $$[a(x,D)u](x)=\int e^{ix\xi}a(x,\xi)\wh{u}(\xi)\,\dbar\xi,\qquad x\in\rz,$$
acting on the Schwarz space $\scrS(\rz^n)$ of rapidly decreasing functions; here, 
$\dbar\xi=(2\pi)^{-n}\,d\xi$. 
The composition of operators $a_0(x,D)$ and $a_1(x,D)$ corresponds to the so-called 
\emph{Leibniz product}  of symbols,  
\begin{align}\label{eq:leibniz-product}
 (a_1\#a_0)(x,\xi)=\iint e^{iy\eta}a_1(x,\xi+\eta)a_0(x+y,\xi)\,dy\dbar\eta. 
\end{align}
$($integration in the sense of oscillatory integrals$)$, cf. for example \cite{Kuma}. 
If the $a_j$ have order $d_j$, then $a_1\#a_0$ has order $d_0+d_1$. The \emph{adjoint symbol} 
 $$a^{(*)}(x,\xi)=\iint e^{-iy\eta}\overline{a(x+y,\xi+\eta;\mu)}\,dy\dbar\eta$$
gives the formal adjoint operator of $a(x,D)$, i.e., 
 $$(a(x,D)u,v)_{L^2}=(u,a^{(*)}(x,D)v)_{L^2},\qquad u,v\in\scrS(\rz^n).$$
If $a(x,\xi;\mu)$ is a symbol that depends on an additional parameter $\mu$, 
we shall write $a(x,D;\mu)$, Leibniz product and adjoint are applied point-wise in $\mu$. 

Throughout the paper we consider a parameter $\mu\in\rpbar:=[0,+\infty)$.

\section{Symbols with finite regularity number} 
\label{sec:03}

\subsection{Grubb's calculus}\label{sec:03.1}

We briefly review a pseudodifferential calculus introduced by Grubb. 
For further details we refer the reader to Chapter 2.1 of \cite{Grub}. 

\begin{definition}\label{def:Grubb01}
By $S^{d,\nu}_{1,0}$ with $d,\nu\in\rz$ 
$($called order and regularity number, respectively$)$ 
denote the space of all symbols $a(x,\xi;\mu)$ satisfying 
\begin{align}\label{eq:estimate01}
|D^\alpha_\xi D^\beta_x D^j_\mu a(x,\xi;\mu)|
\lesssim\spk{\xi}^{\nu-|\alpha|}\spk{\xi,\mu}^{d-\nu-j}+ \spk{\xi,\mu}^{d-|\alpha|-j}.
\end{align}
\end{definition}

The space of smoothing or regularizing symbols, defined as    
\begin{equation}\label{eq:regularizing-symbol}
 S^{d-\infty,\nu-\infty}=\mathop{\mbox{\Large$\cap$}}_{N\ge0}S^{d-N,\nu-N}_{1,0},
\end{equation}
consists of those symbols satisfying, for every $N$ and all orders of derivatives,   
\begin{align*} 
 |D^\alpha_\xi D^\beta_x D^j_\mu a(x,\xi;\mu)|
 &\lesssim\spk{\xi}^{-N} \spk{\mu}^{d-\nu-j}.
\end{align*}

\begin{proposition}
The Leibniz product induces maps 
\begin{align}\label{eq:leibniz00}
S^{d_1,\nu_1}_{1,0}\times S^{d_0,\nu_0}_{1,0}\lra S^{d_0+d_1,\nu}_{1,0},
\qquad \nu=\min(\nu_0,\nu_1,\nu_0+\nu_1). 
\end{align}
\end{proposition}

Asymptotic summations can be performed within the class, in the following sense: 
Given a sequence of symbols $a_\ell\in S^{d-\ell,\nu-\ell}_{1,0}$, there exists an 
$a\in S^{d,\nu}_{1,0}$ such that 
\begin{equation}\label{eq:asymptotic}
 a(x,\xi;\mu)-\sum\limits_{\ell=0}^{N-1} a_\ell(x,\xi;\mu)\in S^{d-N,\nu-N}_{1,0}
\end{equation}
for every $N$; $a$ is uniquely determined modulo $S^{d-\infty,\nu-\infty}$. 


\begin{definition}\label{def:poly}
A symbol $a\in S^{d,\nu}_{1,0}$ is called poly-homogeneous if it satisfies \eqref{eq:asymptotic} 
with $a_\ell\in S^{d-\ell,\nu-\ell}_{1,0}$ that are homogeneous of degree $d-\ell$ in $(\xi,\mu)$ 
for $|\xi|\ge1$.  The space of these symbols is denoted by $S^{d,\nu}$. 
\end{definition} 

If $a\in S^{d,\nu}$, its \textit{homogeneous principal symbol} is defined as  
  $$a^h(x,\xi;\mu):=|\xi|^d a_0\Big(x,\frac{\xi}{|\xi|};\frac{\mu}{|\xi|}\Big)
    =\lim_{t\to+\infty}t^{-d}a(x,t\xi;t\mu),\qquad \xi\not=0.$$
It satisfies 
\begin{align}\label{eq:estimate-hom}
 |D^\alpha_\xi D^\beta_x D^j_\mu a(x,\xi,\mu)|\lesssim |\xi|^{\nu-|\alpha|}|\xi,\mu|^{d-\nu-j}
 +|\xi,\mu|^{d-|\alpha|-j}.
\end{align}

\begin{definition}\label{def:hom-grubb}
$S^{d,\nu}_{hom}$ denotes the space of all smooth functions $a(x,\xi;\mu)$ defined for 
$\xi\not=0$, which are homogeneous of degree $d$ in $(\xi,\mu)$ and satisfy 
$\eqref{eq:estimate-hom}$ for arbitrary orders of derivatives. 
\end{definition}

If $a\in S^{d,\nu}_{hom}$ and $\nu>0$, then $a$ extends by continuity to a function defined for all 
$x$ and $(\xi,\mu)\not=0$; the larger $\nu$ is, the more regular (i.e., differentiable$)$ is this extension.  
This is the justification for the terminology ``regularity number''. In this case we shall identify $a^h$ 
with its extension. 

\begin{definition}\label{def:grubb-elliptic}
Let $\nu>0$. A symbol $a(x,\xi;\mu)\in S^{d,\nu}$ is called elliptic if 
$a^h(x,\xi;\mu)\not=0$ for all $x$ and all $(\xi,\mu)\not=0$ and 
$|a^h(x,\xi,\mu)^{-1}|\lesssim |\xi,\mu|^{-d}$. 
\end{definition}

Note that if $a^h(x,\xi;\mu)$ is constant in $x$ for large $x$, it suffices to require the pointwise 
invertibility of $a^h(x,\xi,\mu)$ 

\begin{theorem}\label{thm:grubb-parametrix}
Let $\nu>0$ and $a\in S^{d,\nu}$. Then $a$ is elliptic if and only if  
there exists a $b\in S^{-d,\nu}$ such that $a\#b-1,b\#a-1\in S^{0-\infty,\nu-\infty}$.  
\end{theorem}
 
Note that if $r\in S^{0-\infty,\nu-\infty}$ with $\nu>0$, then $r(\mu)\xrightarrow{\mu\to+\infty}0$ in 
$S^{-\infty}(\rz^n)$. In particular, if $a$ is elliptic then $a(x,D;\mu)$ is invertible for 
large $\mu$. 

\eqref{eq:estimate01} and \eqref{eq:estimate-hom} suggest to introduce two subspaces
of $S^{d,\nu}_{1,0}$ and $S^{d,\nu}_{hom}$, respectively, with estimates corresponding 
to the first and second term on the right-hand side, respectively. 
These will be discussed in the next two subsections.

\subsection{Strong parameter-dependence (symbols of infinite regularity)}
\label{sec:03.2}

In this section we consider the space 
$S^d_{1,0}=\mathop{\mbox{\Large$\cap$}}_{N\ge0}S^{d,N}_{1,0}$ and the 
poly-homogeneous subclass. For clarity we prefer to present it in an independent way. 

\begin{definition}\label{def:Hoer01}
$S^d_{1,0}$ consists of all symbols $a(x,\xi;\mu)$ satisfying, for all orders of derivatives, 
 $$|D^\alpha_\xi D^\beta_x D^j_\mu a(x,\xi;\mu)|\lesssim \spk{\xi,\mu}^{d-|\alpha|-j}.$$
\end{definition}

We shall call such symbols also \textit{strongly parameter-dependent}, since differentiation 
with respect to $\xi$ or $\mu$ improves the decay in $(\xi,\mu)$. 

The space of regularizing symbols $S^{-\infty}=\cap_{d\in\rz} S^d$ consists of those symbols 
which are rapidly decreasing in $(\xi,\mu)$ and $\scrC^\infty_b$ in $x$. 

\begin{proposition}
The Leibniz product induces maps $S^{d_1}_{1,0}\times S^{d_0}_{1,0}\to S^{d_0+d_1}_{1,0}$. 
\end{proposition}

\begin{definition}\label{def:strongly-hom}
$S^d_{hom}$ consist of all symbols $a(x,\xi;\mu)$ defined for $(\xi,\mu)\not=0$ which 
are homogeneous of degree $d$ in $(\xi,\mu)$ and satisfy, for every order of derivatives, 
\begin{equation}\label{eq:stima}
|D^\alpha_\xi D^\beta_x D^j_\mu a(x,\xi;\mu)|\lesssim |\xi,\mu|^{d-|\alpha|-j}  
\end{equation}
\end{definition}

\forget{
\begin{remark}
Initially one could also ask that the homogeneous components from 
Definition $\ref{def:strongly-hom}$ are defined only for $\xi\not=0$. 
However, since the restriction to any of the domains 
$Q_n=\{(\xi,\mu)\mid 0<|\xi|<1,\,1/n<\mu<n\}$, $n\in\nz$, is smooth 
and has bounded derivatives of any order, the function extends to a smooth 
function on the closure of $Q_n$. Thus the homogeneous component has a 
smooth extension to $\rz^n\times (\rz^n\times\rpbar)\setminus\{0\}$. 
With this observation in mind, it holds 
$S^d_{hom}=\mathop{\mbox{\Large$\cap$}}_{\nu\ge0}S^{d,\nu}_{hom}$. 
\end{remark}
}

\begin{definition}\label{def:Hoer-poly}
A symbol $a\in S^d_{1,0}$ is called poly-homogeneous if there exists a sequence of 
homogeneous symbols $a_{\ell}\in S^{d-\ell}_{hom}$ such that, for every $N$,  
 $$a(x,\xi;\mu)-\sum_{\ell=0}^{N-1} \chi(\xi,\mu) a_\ell(x,\xi;\mu)\in S^{d-N}_{1,0},$$ 
where $\chi(\xi,\mu)$ is an aribtrary zero-excision function. The space of such symbols 
will be denoted by $S^d$. 
\end{definition}

We call $a_0$ the homogeneous principal symbol of $a\in S^d$, and 
 $$a_0(x,\xi;\mu)=\lim_{t\to+\infty}t^{-d}a(x,t\xi;t\mu),\qquad (\xi,\mu)\not=0.$$
Ellipticity of $a$ is defined as in Definition \ref{def:grubb-elliptic} and the obvious analog of 
Theorem \ref{thm:grubb-parametrix} is valid. 

\begin{remark}
In the literature, the space $S^d$ is frequently denoted by  
$S^d_{\mathrm{cl}}$ and the symbols are called \emph{classical} rather than poly-homogeneous. 
\end{remark}

\subsection{Weakly parameter-dependent symbols}\label{sec:03.3}

Let us describe the second natural subspace of $S^{d,\nu}_{1,0}$. 

\begin{definition}\label{def:weak01}
Let $\wt{S}^{d,\nu}_{1,0}$ denote the space of all symbols $a(x,\xi;\mu)$ which satisfy, 
for every order of derivatives, 
\begin{align*}
 |D^\alpha_\xi D^\beta_x D^j_\mu a(x,\xi;\mu)|
 \lesssim \spk{\xi}^{\nu-|\alpha|}\spk{\xi,\mu}^{d-\nu-j}. 
\end{align*}
\end{definition}

Note that $\wt{S}^{d,\nu}_{1,0}=S^{d,\nu}_{1,0}$ whenever $\nu\le0$. In particular, 
$\wt{S}^{d-\infty,\nu-\infty}_{1,0}=S^{d-\infty,\nu-\infty}_{1,0}$

\begin{proposition}
The Leibniz product induces maps 
$\wt{S}^{d_1,\nu_1}_{1,0}\times \wt{S}^{d_0,\nu_0}_{1,0}\to\wt{S}^{d_0+d_1,\nu_0+\nu_1}_{1,0}$. 
\end{proposition}

\begin{definition}\label{def:weak02}
Let $\wt{S}^{d,\nu}_{hom}$ denote the space of all functions $a(x,\xi;\mu)$ which are defined for 
$\xi\not=0$, are homogeneous in $(\xi,\mu)$ of degree $d$ and satisfy, for every order of derivatives,
\begin{align}\label{eq:estimate-hom-tilde}
 |D^\alpha_\xi D^\beta_x D^j_\mu a(x,\xi;\mu)|\lesssim |\xi|^{\nu-|\alpha|}|\xi,\mu|^{d-\nu-j}.
\end{align}
\end{definition}

\begin{definition}\label{def:weak03}
A symbol $a\in \wt{S}^{d,\nu}_{1,0}$ is called poly-homogeneous if there exists a sequence of 
homogeneous symbols $a_{\ell}\in \wt{S}^{d-\ell,\nu-\ell}_{hom}$ such that, for every $N$,  
 $$a(x,\xi;\mu)-\sum_{\ell=0}^{N-1} \chi(\xi) a_\ell(x,\xi;\mu)\in \wt{S}^{d-N,\nu-N}_{1,0},$$ 
where $\chi(\xi)$ is an arbitrary zero-excision function. The space of such symbols 
will be denoted by $\wt{S}^{d,\nu}$. 
\end{definition}

Again, $a_0$ is called the homogeneous principal symbol of $a\in\wt{S}^{d,\nu}$, and   
 $$a_0(x,\xi;\mu)=\lim_{t\to+\infty}t^{-d}a(x,t\xi;t\mu),\qquad \xi\not=0.$$

\begin{definition}\label{def:tilde-elliptic}
A symbol $a(x,\xi;\mu)\in \wt{S}^{d,\nu}$ is called elliptic if $a_0(x,\xi;\mu)\not=0$ for all 
$x$, $\mu$ and all $\xi\not=0$,  and $|a_0(x,\xi,\mu)^{-1}|\lesssim |\xi|^{-\nu}|\xi,\mu|^{-d+\nu}$. 
\end{definition}

\begin{theorem}\label{thm:tilde-parametrix}
A symbol $a\in \wt{S}^{d,\nu}$ is elliptic if and only if there exists a $b\in\wt{S}^{-d,-\nu}$ 
such that $a\#b-1,b\#a-1\in \wt{S}^{0-\infty,0-\infty}$.  
\end{theorem}

Note that ellipticity of $a\in \wt{S}^{d,\nu}$ is not equivalent to the point-wise invertibility 
of the homogeneous principal symbol $a_0$ on its domain, even not in case of independence 
of the $x$-variable (see Theorem \ref{thm:main01} and the subsequent comment). 
Moreover, a remainder $r\in \wt{S}^{0-\infty,0-\infty}$ is, in general, only bounded but not 
decaying as $\mu\to+\infty$. Therefore $a(x,D;\mu)$ need not be invertible for large $\mu$.  

\section{Regularity number and weighted spaces}
\label{sec:04}

In any of the so far introduced symbol spaces, the involved variable $x$ enters as a 
$\scrC^\infty_b$-variable, while the spaces differ by the structures in the variables 
$(\xi,\mu)$. For this reason, and also to keep notation more lean, in this section we 
ignore the $x$-dependence and focus on symbols depending only on $(\xi,\mu)$.  

Let us denote by $\sz^n_+$ the unit semi-sphere, 
\begin{align}\label{eq:unit-sphere}
 \sz^n_+:=\big\{(\xi,\mu)\in\rz^n\times\rpbar \mid |\xi|^2+\mu^2=1\big\}. 
\end{align}
Every homogeneous symbol $a\in S^d_{hom}$ is of the form
\begin{align}\label{eq:hom_ext}
a(\xi;\mu)=|\xi,\mu|^d\,\wh{a}\Big(\frac{(\xi,\mu)}{|\xi,\mu|}\Big),\qquad 
\wh{a}=a|_{\sz^n_+}\in\scrC^\infty(\sz^n_+),
\end{align}
and the map $a\mapsto \wh{a}$ establishes an isomorphism between $S^d_{hom}$ and 
$\scrC^\infty(\sz^n_+)$. A symbol $a\in\wt{S}^{d,\nu}_{hom}$ is defined for $\xi\not=0$ only, 
hence its restriction is defined only on the punctured unit semi-sphere 
\begin{align}\label{eq:punctured-sphere}
  \whsz^n_+:=\sz^n_+\setminus\{(0,1)\}=\{(\xi,\mu)\in\sz^n_+\mid \xi\not=0\}.
\end{align}
We shall now investigate,  which subspace of $\scrC^\infty(\whsz^n_+)$ is in 1-1-correspondence 
with  $\wt{S}^{d,\nu}_{hom}$. To this end, we shall identify 
$\whsz^n_+$ with $(0,1]\times \sz^{n-1}$, using the (polar-)coordinates 
 $$\xi=r\phi,\qquad \mu=\sqrt{1-r^2}, \qquad (0<r\le1,\quad \phi\in\sz^{n-1}).$$

If $E$ is an arbitrary Fr\'echet space, we shall denote by 
$\scrC^\infty_{B}((0,\eps),E)$ the space of all smooth bounded functions $u:(0,\eps)\to E$ such that 
$(r\partial_r)^\ell u$ is bounded on $(0,\eps)$ for every order of derivatives. 

\begin{definition}
With $\gamma\in\rz$ define  
\begin{align*}
\scrC^{\infty}_{B}(\whsz^n_+)
 &:=\big\{a\in\scrC^\infty(\whsz^n_+)\mid 
     a(r,\phi)\in\scrC^{\infty}_{B}\big((0,\eps),\scrC^\infty(\sz^{n-1})\big)
     \text{ for some }\eps>0\big\},\\
 r^\gamma\,\scrC^{\infty}_{B}(\whsz^n_+)
 &:=\big\{a\in\scrC^\infty(\whsz^n_+)\mid 
     r^{-\gamma}a\in \scrC^{\infty}_{B}(\whsz^n_+)\big\}.
\end{align*}
\end{definition}

In other words, the index $\gamma$ indicates  the rate of (non-)vanishing in the point 
$(\xi,\mu)=(0,1)$; we shall also speak of spaces with weight $\gamma$. Note that $|\xi|=r$. 

\begin{definition}\label{def:weighted-hom}
Let $\wt{S}^{(d,\gamma)}$ denote the space of all functions $a(\xi;\mu)$ defined for $\xi\not=0$ 
of the form 
 $$a(\xi;\mu)=|\xi,\mu|^d \,\wh{a}\Big(\frac{(\xi,\mu)}{|\xi,\mu|}\Big),\qquad 
    \wh{a}\in r^\gamma\scrC^{\infty}_{B}(\whsz^n_+).$$
\end{definition}

Let $a$ and $\wh{a}$ be as in the previous definition. Identifying $\wh{a}(\xi,\mu)$ with its local 
representation $\wh{a}(r,\phi)$, we have the relations 
\begin{align}\label{eq:localcoord}
 \wh{a}(r,\phi)=a\left(r\phi;\sqrt{1-r^2}\right),\qquad 
 a(\xi;\mu)=|\xi,\mu|^d \,\wh{a}\Big(\frac{|\xi|}{|\xi,\mu|},\frac{\xi}{|\xi|}\Big).
\end{align} 
In particular, the $d$-homogeneous extension of $\wh{a}(r,\phi)=r^\nu$ is 
$a(\xi;\mu)=|\xi|^{\nu}|\xi,\mu|^{d-\nu}$. 

\begin{lemma}\label{lem:localized}
Let $\wh\chi\in\scrC^\infty(\sz^n_+)$ vanish in a small neighborhood of $(\xi,\mu)=(0,1)$ and 
let $\chi(\xi,\mu)=\wh{\chi}((\xi,\mu)/|\xi,\mu|)\in S^0_{hom}$ be its homogeneous extension 
of degree $0$. Then 
 $$\chi S^d_{hom}=\chi{S}^{d,\nu}_{hom}=\chi \wt{S}^{(d,\nu)}.$$
\end{lemma}

In fact, it suffices to observe that $\wh{\chi}(\xi,\mu)$ is supported in a set of the form 
$\{(\xi,\mu)\mid 0\le\mu\le c|\xi|\}$ on which $|\xi|\le|\xi,\mu|\lesssim|\xi|$. 

The following theorem shows that, for weakly parameter-dependent homogeneous components, 
regularity number and weight are the same thing. 

\begin{theorem}\label{thm:main01}
$\wt{S}^{(d,\nu)}=\wt{S}^{d,\nu}_{hom}$ for every $d,\nu\in\rz$. 
In particular, the map $a\mapsto a|_{\whsz^n_+}$ establishes an isomorphism between 
$\wt{S}^{d,\nu}_{hom}$ and $r^\nu\scrC^{\infty}_{B}(\whsz^n_+)$. 
\end{theorem}	
\begin{proof}
Let us first prove the inclusion ``$\subseteq$''. Let $a(\xi;\mu)$ be as in 
Definition \ref{def:weighted-hom}. 
By multiplication with $|\xi,\mu|^{-d}$, we may assume without loss of generality that $d=0$. 
In view of Lemma \ref{lem:localized} we may assume that $\wh{a}$ is supported in a small 
neighborhood of $(0,1)$. Hence, in the representation \eqref{eq:localcoord} we may assume that 
$\wh{a}(r,\phi)\in r^\nu\scrC^\infty_B((0,1)\times\sz^{n-1})$ 
vanishes for $r\ge \delta$ for some $\delta<1$. By a standard tensor-product 
argument\footnote{If $E$ is some Fr\'echet space, then 
$\scrC^\infty(\sz^n,E)=\scrC^\infty(\sz^n)\wh{\otimes}_\pi E$ as a completed projective 
tensor-product of Fr\'echet spaces. Thus any function $a\in \scrC^\infty(\sz^n,E)$ can be 
written as an infinite sum $\sum_{j=1}^\infty \lambda_j \omega_j e_j$ with zero-sequences
$(\omega_j)\subseteq\scrC^\infty(\sz^n)$, $(e_j)\subset E$, and an absolutely summable 
numerical sequence $(\lambda_j)$.}
we can assume that $\wh{a}$ is of the form  
 $$\wh{a}(r,\phi)=u(r)\omega(\phi),\qquad u\in r^\nu\scrC^\infty_B((0,1)),\quad 
   \omega\in \scrC^\infty(\sz^{n-1}),$$
where $u$ is supported in $(0,\delta]$.  
We also may assume $\nu=0$, since the homogeneous extension of degree $d=0$ of $r^{\nu}$ 
is just $|\xi|^\nu |\xi,\mu|^{-\nu}$. Summing up, we can assume $d=\nu=0$ and 
 $$a(\xi;\mu)=u\left(\frac{|\xi|}{|\xi,\mu|}\right)\omega\left(\frac{\xi}{|\xi|}\right).$$
By induction, it is then straightforward to verify that $D^\alpha_\xi D^j_\mu a(\xi;\mu)$ 
is a finite linear combination of terms of the form 
 $$((rD_r)^ku)\left(\frac{|\xi|}{|\xi,\mu|}\right)p_{j+\ell}(\xi,\mu)q_m(\xi),
    \quad \ell,m\in\nz_0,\quad \ell+m=|\alpha|,$$
with $p_{j+\ell}\in S^{-(j+\ell)}_{hom}$ and $q_m$ being smooth and homogeneous of degree 
$-m$ in $\xi\not=0$. 
This gives immediately the estimate 
 $$|D^\alpha_\xi D^j_\mu a(\xi;\mu)|
 \lesssim \sum_{\substack{\ell+m=|\alpha|\\ \ell,m\in\nz_0}}|\xi|^{-m}|\xi,\mu|^{-j-\ell}
 \lesssim |\xi|^{-|\alpha|}|\xi,\mu|^{-j}.$$
Next we shall show the inclusion ``$\supseteq$''. Let $a\in\wt{S}^{d,\nu}_{hom}$ be given.  
It is enough to consider the case $d=\nu=0$, since 
$a\in \wt{S}^{d,\nu}_{hom}$ if and only if 
$|\xi|^{-\nu}|\xi,\mu|^{\nu-d}a\in \wt{S}^{0,0}_{hom}$ and 
$|\xi|^{-\nu}|\xi,\mu|^{\nu-d}=r^{-\nu}$ in polar-coordinates. 
Again, $a$ can be assumed do have support in a small conical neighborhood containing $(0,1)$. Thus 
 $$\wh{a}(r,\phi)=a\left(r\phi,\sqrt{1-r^2}\right)
     =a\left(\phi;v(r)\right),\qquad v(r)=\frac{\sqrt{1-r^2}}{r},$$ 
vanishes for $r\ge\delta$ for some $\delta<1$. 
Extend $\wh{a}$ from $(0,1)\times\sz^{n-1}$ to $(0,1)\times(\rz^n\setminus\{0\})$ by 
 $$\wh{a}(r,\phi)=a\left(\frac{\phi}{|\phi|};v(r)\right),\qquad 0\not=\phi\in\rz^n.$$
Using that $r v^\prime(r)/v(r)=1/(r^2-1)$, it  is straightforward to see that 
$(r\partial_r)^\ell \partial^\alpha_\phi \wh{a}(r,\phi)$ is a linear combination of terms of 
the form 
 $$((\mu\partial_\mu)^j\partial^\beta_\xi a)\left(\frac{\phi}{|\phi|},v(r)\right)q(\phi)g(r),
     \qquad j\le \ell,\quad\beta\le\alpha,$$
where $q$ is  smooth and homogeneous of degree $-|\alpha|$ in $\phi\not=0$ and 
$g\in\scrC^\infty([0,1))$. Thus $(r\partial_r)^\ell \partial^\alpha_\phi \wh{a}(r,\phi)$ is 
bounded for $r\in(0,\delta]$ and $\phi$ belonging to a small neighborhood of the unit-sphere 
$\sz^n$. This shows the claim. 
\end{proof}

In particular, we see that $\wt{S}^{d,\nu}_{hom}$ does not behave well under inversion: if 
$a\in r^\nu\scrC^\infty(\whsz^n_+)$ is point-wise invertible, the inverse will, in general, not belong 
to such a weighted space. To guarantee this, an additional control at the singularity of $a$ is needed. 
This will be addressed in the sequel. 

\begin{theorem}\label{thm:main02}
$S^{d,\nu}_{hom}=\wt{S}^{d,\nu}_{hom}+S^{d}_{hom}$
for every $d,\nu\in\rz$. 
\end{theorem}
\begin{proof}
The first identity is true in case $\nu\le0$, since then 
$S^{d}_{hom}\subseteq{S}^{d,\nu}_{hom}=\wt{S}^{d,\nu}_{hom}$ 
by definition of the involved spaces. 

It remains to consider $\nu>0$. The inclusion $\supseteq$ is clear. 
By multiplication with $|\xi,\mu|^{-d}$ we may assume without loss of generality that $d=0$. 

Let $a\in S^{0,\nu}_{hom}$ be given. We use Theorem \ref{thm:main01} and show that 
the restriction of $a$ to $\sz^n_+$ is the sum of a smooth function and a function belonging to 
$r^{\nu}\scrC^\infty_B(\whsz^n_+)$. 
By Lemma \ref{lem:localized} it suffices to find a decomposition for $(1-\chi)a$. 

Let $N$ be the largest natural number with $N<\nu$. 
It can be shown $($see Lemma 2.1.10 and Proposition 2.1.11 in \cite{Grub}$)$ 
that $a$ extends as an $N$-times continuously differentiable function to 
$\rz^n\times\rpbar\setminus\{(0,0)\}$ and if $p_N(\xi;\mu)$ denotes the Taylor-polynomial 
of $a$ in $\xi$ around $\xi=0$, then $p_N$ is smooth in $\mu>0$ and 
$$r_N(\xi;\mu):= a(\xi;\mu)-p_N(\xi;\mu) =\sum_{|\alpha|=N+1}\frac{N+1}{\alpha!}
\xi^\alpha\int_0^1 (1-t)^N (\partial^\alpha_\xi a)(t\xi;\mu)\,dt.$$
Since $(1-\chi)p_N$ is smooth on $\sz^n_+$, it remains to verify that the restriction of 
$(1-\chi)r_N$ belongs to $r^{\nu}\scrC^\infty_B(\whsz^n_+)$. To this end let
 $$r_\alpha(\xi,\mu)
   ={\xi}^\alpha\int_0^1 (1-t)^N \partial^\alpha_\xi a(t\xi;\mu)\,dt,\qquad|\alpha|=N+1.$$
Then, in polar-coordinates,  
 $$\wh{r}_\alpha(r,\phi)
   =\phi^\alpha\int_0^1 (1-t)^N (\partial^\alpha_\xi a)(t\phi;v(r))\,dt,\qquad 
   v(r)=\frac{\sqrt{1-r^2}}{r}.$$
It suffices to show that $\wh{r}_\alpha\in r^\nu\scrC^\infty_B((0,\eps),\scrC^\infty(\sz^{n-1}))$ for 
some $\eps>0$. We have 
 $$|\wh{r}_\alpha(r,\phi)|\le \int_0^1|t\phi|^{\nu-|\alpha|}|t\phi,v(r)|^{-\nu}\,dt
   =r^{\nu}\int_0^1 t^{-1+(\nu-N)}|tr\phi,\sqrt{1-r^2}|^{-\nu}\,dt.$$
Since $|tr\phi,\sqrt{1-r^2}|^{-\nu}\lesssim1$ for $r\le\delta$. we find that 
$r^{-\nu}|\wh{r}_\alpha(r,\phi)|$ is bounded. 
Derivatives of $\wh{r}_\alpha$ are treated similarly, proceeding as in the proof of 
Theorem \ref{thm:main01}.  
\end{proof}

This decomposition also shows how to associate with a symbol $a\in S^{d,\nu}_{hom}$ a symbol 
$p\in S^{d,\nu}$ with homogeneous principal symbol equal to $a$. 
In fact, writing $a=\wt{a}+a_{\mathrm{smooth}}$ 
with $a\in\wt{S}^{d,\nu}_{hom}$ and $a_{\mathrm{smooth}}\in S^d_{hom}$, choose 
 $$p(\xi;\mu)=\wt\chi(\xi)\wt{a}(\xi;\mu)+\chi(\xi,\mu) a_{\mathrm{smooth}}(\xi;\mu)$$
with arbitrary zero-excision functions $\chi(\xi,\mu)$ and $\wt\chi(\xi)$. 
Changing the cut-off functions induces remainders in $S^{d-\infty,\nu-\infty}$; 
hence we may assume that 
$\wt\chi(\xi)\chi(\xi,\mu)=\wt\chi(\xi)$ and $p=\wt{\chi}a+(1-\wt{\chi})\chi a_{\mathrm{smooth}}$. 
Then taking another representation 
$a=\wt{a}^\prime+a_{\mathrm{smooth}}^\prime$ with associated symbol $p^\prime$, we find 
 $$p-p^\prime=(1-\chi)(\xi)\kappa(\xi,\mu)(a_{\mathrm{smooth}}-a^\prime_{\mathrm{smooth}}).$$
Noting that $($after restriction to the unit-sphere$)$
$${a}_{\mathrm{smooth}}-\wt{a}^\prime_{\mathrm{smooth}}=r^{n(\nu)}b_{\mathrm{smooth}},
\qquad n(\nu)=\text{smallest integer $\ge\nu$},$$ 
with a function $b\in\scrC^\infty(\sz^n_+)$ one concludes that $p-p^\prime$ belongs to 
$S^{d-\infty,n(\nu)-\infty}\subseteq S^{d-\infty,\nu-\infty}$. 

In combination with Lemma \ref{lem:localized} we obtain the following: 

\begin{theorem}\label{thm:decomposition}
Let $V=\{(\xi,\mu) \mid \mu\ge c|\xi|\}$ with some constant $c\ge0$. Then 
	$${S}^{d,\nu}=\wt{S}^{d,\nu}_V+S^{d},$$
where $\wt{S}^{d,\nu}_V\subset \wt{S}^{d,\nu}$ is the subspace of those symbols 
whose homogeneous components have support in $V$. 
\end{theorem}

\section{Expansion at infinity}
\label{sec:05}

One of the motivations for this paper is to extend the concept of ellipticity in the spaces 
$S^{d,\nu}$ with positive regularity number $\nu$ to the case $\nu=0$. 
Ellipticity should still be characterized by the invertibility of one or more principal symbols 
$($plus some uniformity assumptions for preserving the $\scrC^\infty_b$ structure in $x)$ 
and should imply invertibility of $a(x,D;\mu)$ for large values of the parameter $\mu$. 
Recall that $S^{d,0}=\wt{S}^{d,0}$; for systematic reasons we address this 
question in $\wt{S}^{d,\nu}$ for arbitrary $\nu$. 

In a first step, in Section \ref{sec:05.1}, we introduce a subclass $\wtbfS^{d,\nu}_{1,0}$ of 
$\wt{S}^{d,\nu}_{1,0}$ in which elliptic elements are invertible for large values of $\mu$. 
The ellipticity involves an estimate of the full symbol and the invertibilty of a so-called limit-symbol; 
the latter plays the role of a new principal symbol. In a second step we pass to the subclass of 
poly-homogeneous symbols $\wtbfS^{d,\nu}$ where the full symbol can be replaced by the 
homogeneous principal symbol. 

\subsection{Symbols with expansion at infinity}
\label{sec:05.1}

\begin{definition}\label{def:symbol-expansion}
We denote by $\wtbfS^{d,\nu}_{1,0}$, $d,\nu\in\rz$, 
the subspace of $\wt{S}^{d,\nu}_{1,0}$ consisting of all symbols $a(x,\xi;\mu)$ for which 
exists a sequence of symbols $a^\infty_{[\nu+j]}\in S^{\nu+j}_{1,0}(\rz^n)$, 
$j\in\nz_0$, such that  
    $$r_{a,N}(x,\xi;\mu):=a(x,\xi;\mu)
      -\sum_{j=0}^{N-1}a_{[\nu+j]}^\infty(x,\xi)[\xi,\mu]^{d-\nu-j}
      \in \wt{S}^{d,\nu+N}_{1,0}$$ 
for every $N\in\nz_0$; here $[\xi,\mu]\in S^1$ denotes a smooth positive function that coincides 
with $|(\xi,\mu)|$ outside some compact set. The symbol $a^\infty_{[\nu]}$ shall be called the 
\emph{principal limit-symbol} of $a$. 
\end{definition}

The definition does not depend on the choice of the function $[\xi,\mu]$, since the difference 
of two such functions belongs to $\scrC^\infty_{\mathrm{comp}}(\rpbar\times\rz^n)$; 
for a further discussion see Section \ref{sec:07.2} below. 
The coefficients $a_{[\nu+j]}^\infty(x,\xi)$ are uniquely determined by $a$. 
$\wtbfS^{d,\nu}_{1,0}$ is a Fr\'echet space when equipped with the projective topology with respect to 
the mappings 
 $$a\mapsto r_{a,N}:\wtbfS^{d,\nu}_{1,0}\lra \wt{S}^{d,\nu+N}_{1,0},\qquad 
    a\mapsto a_{[\nu+j]}^\infty:\wtbfS^{d,\nu}_{1,0}\lra S^{\nu+j}_{1,0}(\rz^n).$$
Note that $\wtbfS^{d-N,\nu-N}_{1,0}\subset \wtbfS^{d,\nu}_{1,0}$ whenever $N\in\nz$; we define 
 $$\wtbfS^{d-\infty,\nu-\infty}_{1,0}
   =\mathop{\mbox{\large$\cap$}}_{N\in\nz}\wtbfS_{1,0}^{d-N,\nu-N}.$$
Obviously, the maps 
  \begin{align*}
   a\mapsto \spk{\xi}^e a:\wtbfS^{d,\nu}_{1,0}\lra \wtbfS^{d+e,\nu+e}_{1,0},\qquad 
   a\mapsto [\xi,\mu]^e a:\wtbfS^{d,\nu}_{1,0}\lra \wtbfS^{d+e,\nu}_{1,0}, 
  \end{align*}
are isomorphisms with $(\spk{\xi}^e a)_{[\nu+e+j]}^\infty=\spk{\xi}^e a_{[\nu+j]}^\infty$ and 
$([\xi,\mu]^e a)_{[\nu+j]}^\infty=a_{[\nu+j]}^\infty$. 

\begin{example}\label{ex:without-parameter}
Let $a(x,\xi)\in S^d_{1,0}(\rz^n)$ be independent of $\mu$. 
Then $a\in\wtbfS^{d,d}_{1,0}$ with $a_{[d]}^\infty=a$ and $a_{[d+j]}^\infty=0$ 
for every $j\ge1$. 
\end{example}

\begin{proposition}\label{prop:classical}
Let $a\in S^d$. Then $a\in\wtbfS^{d,0}_{1,0}$ with 
principal limit-symbol  
 $$a_{[0]}^\infty(x,\xi)=a_{0}(x,0;1),$$
i.e., the homogeneous principal symbol of $a$ evaluated in $(\xi,\mu)=(0,1)$. Moreover, 
$a_{[j]}^\infty(x,\xi)$ is a homogeneous polynomial in $\xi$ of order $j$. 
\end{proposition}

Note that the proof of Proposition \ref{prop:classical} is constructive, i.e., for given $a$ all 
symbols $a_{[j]}^\infty(x,\xi)$ can be calculated explicitly.  
\forget{
For example, 
if $a(\xi;\mu)=\mu^{d}$ with $d\in\nz$, then $a^\infty_{[j]}(\xi)$ is the $j$-th term in the 
Taylor-series of $a(\xi;\sqrt{1-|\xi|^2})=(1-|\xi|^2)^{d/2}$ centered in $0$.  
}

\begin{proof}[Proof of Proposition $\ref{prop:classical}$]
For convenience assume independence on the $x$-variable. 
First note that $S^{d-N}\subseteq \wt{S}^{d-N,0}_{1,0}\subseteq \wt{S}^{d,N}_{1,0}$, since 
 $$\spk{\xi,\mu}^{d-N-|\alpha|-j}\le\spk{\xi}^{-|\alpha|}\spk{\xi,\mu}^{d-N-j}
     \le \spk{\xi}^{N-|\alpha|}\spk{\xi,\mu}^{d-N-j}.$$  
Thus we may assume that $a(\xi;\mu)=\chi(\xi,\mu)a_\ell(\xi;\mu)$ with 
$a_\ell\in S^{d-\ell}_{hom}$, $\ell\ge 0$, and a zero-excision function $\chi(\xi,\mu)$, 
and to show that $a$ it belongs to 
$\wtbfS^{d-\ell,0}_{1,0}\subseteq\wtbfS^{d,0}_{1,0}$. 

Let $\wh\kappa\in\scrC^\infty(\sz^n_+)$ be supported close to $(0,1)$ and $\kappa\equiv 1$ 
near $(0,1)$ and define $\kappa(\xi;\mu):=\wh\kappa((\xi,\mu)/|\xi,\mu|)$. 

\textbf{Step 1:} $1-\kappa$ is supported in a conical set $V$ of the form 
$V=\{(\xi,\mu)\mid |\xi|\ge c\mu\}$ with $c>0$. Therefore 
$(1-\kappa)(\xi;\mu)a(\xi;\mu)\in \wt{S}^{d-\ell,L}$ for every $L$,  
since $\spk{\xi}\sim\spk{\xi,\mu}$ on its support, hence 
$\spk{\xi,\mu}^{d-\ell-|\alpha|-j}\sim\spk{\xi}^{L-|\alpha|}\spk{\xi,\mu}^{d-\ell-L-j}$.

\textbf{Step 2:} Assume that $a_\ell|_{\sz^n_+}$ vanishes to order $N$ in $(\xi,\mu)=(0,1)$. 
Then 
 $$u(\xi):=(\kappa a_\ell)\big(\xi;\sqrt{1-|\xi|^2}\big)$$
is a smooth function with compact support in $B:=\{\xi\mid|\xi|<1\}$ that vanishes to order 
$N$ in $\xi=0$. Write 
$u(\xi)=\sum\limits_{|\alpha|=N}\xi^\alpha u_\alpha(\xi)$ with $u_\alpha$ also 
compactly supported in $B$. Then 
\begin{align*}
 (\kappa a_\ell)(\xi;\mu)
 =|\xi,\mu|^{d-\ell} u(\xi/|\xi,\mu|)
 =|\xi,\mu|^{d-\ell-N}\sum_{|\alpha|=N}\xi^\alpha u_\alpha(\xi/|\xi,\mu|)
\end{align*}
and therefore 
 $$(\kappa a)(\xi;\mu)= \sum_{|\alpha|=N}\xi^\alpha p_\alpha(\xi;\mu),\qquad 
    p_\alpha\in S^{d-\ell-N}.$$
Hence $(\kappa a)(\xi;\mu)\in \wt{S}^{d-\ell,N}_{1,0}$ and thus, 
by Step 1,  $a(\xi;\mu)\in \wt{S}^{d-\ell,N}_{1,0}$.  

\textbf{Step 3:} Let 
$p(\xi;\mu)=\sum\limits_{|\alpha|\le N-1}u_\alpha\xi^\alpha[\xi,\mu]^{d-\ell-|\alpha|}$ 
where $u_\alpha$ is the $\alpha$-th Taylor coefficient of 
$a_{\ell}\big(\xi;\sqrt{1-|\xi|^2}\big)$ in $\xi=0$. 
Then $p\in S^{d-\ell}$ is homogeneous of degree $d-\ell$ for $|\xi,\mu|\ge1$; let 
$p_\ell\in S^{d-\ell}_{hom}$ be the homogeneous principal symbol. Write 
 $$a-p=\chi(a_{\ell}-p_{\ell})-(p-\chi p_\ell)=\chi(a_{\ell}-p_{\ell})
     \mod S^{-\infty}.$$
Since $(a_{\ell}-p_\ell)|_{\sz^n_+}$ vanishes to order $N$ in $(0,1)$, we conclude by Step 2 
that  $a-p\in \wt{S}^{d-\ell,N}_{1,0}$. Hence 
 $$a(\xi;\mu)\equiv \sum_{j=0}^{N-1}a^\infty_{[j]}(\xi)[\xi,\mu]^{d-\ell-j},\qquad 
     a^\infty_{[j]}(\xi)=\sum_{|\alpha|=j}u_\alpha\xi^\alpha,$$
modulo $\wt{S}^{d-\ell,N}_{1,0}$. 
\end{proof}

\forget{
\begin{lemma}
If $a_k\in \wtbfS^{d_k,\nu_k}_{1,0}$ for $k=0,1$, then 
$a_1a_0\in \wtbfS_{1,0}^{d_0+d_1,\nu_0+\nu_1}$ with 
 $$(a_1a_0)_{[\nu_0+\nu_1+j]}^\infty
    =\sum_{k+\ell=j}a_{1,[\nu_1+\ell]}^\infty a_{0,[\nu_0+k]}^\infty.$$
\end{lemma}
\begin{proof}
Obviously, $a:=a_1a_0\in\wt{S}^{d_0+d_1,\nu_0+\nu_1}_{1,0}$. 
By a straightforward calculation, 
\begin{align*} 
 \Big(\sum_{j=0}^{N-1}[\xi,\mu]^{d_1-\nu_1-j}&a_{1,[\nu_1-j]}^\infty\Big) 
   \Big(\sum_{j=0}^{N-1}[\xi,\mu]^{d_0-\nu_0-j}a_{0,[\nu_0+j]}^\infty\Big)\\
   &\equiv \sum_{j=0}^{N-1}[\xi,\mu]^{d_0+d_1-(\nu_0+\nu_1)-j}a_{[\nu_0+\nu_1+j]}^\infty
   \mod \wt{S}^{d_0+d_1,\nu_0+\nu_1+N}_{1,0}. 
\end{align*} 
Thus we find, modulo $\wt{S}^{d_0+d_1,\nu_0+\nu_1+N}_{1,0}$, 
 $$a-\sum_{j=0}^{N-1}[\xi,\mu]^{d_0+d_1-(\nu_0+\nu_1)-j}a_{[\nu_0+\nu_1+j]}^\infty
   \equiv a_1 r_{a_0,N}+r_{a_1,N}\sum_{j=0}^{N-1}[\xi,\mu]^{d_0-\nu_0-j}a_{0,[\nu_0+j]}^\infty.$$
Both summands on the right-hand side belong to $\wt{S}^{d_0+d_1,\nu_0+\nu_1+N}_{1,0}$. 

The continuity follows from the closed graph theorem, since multiplication is a continuous map
$\wt{S}^{d_1,\nu_1}_{1,0}\times\wt{S}^{d_0,\nu_0}_{1,0}\lra 
\wt{S}_{1,0}^{d_0+d_1,\nu_0+\nu_1}$. 
\end{proof}
} 

\begin{lemma}
The following holds true$:$ 
\begin{itemize}
\item[i$)$ ]If $a_k\in \wtbfS^{d_k,\nu_k}_{1,0}$ for $k=0,1$, then 
 $a_1a_0\in \wtbfS_{1,0}^{d_0+d_1,\nu_0+\nu_1}$ with 
  $$(a_1a_0)_{[\nu_0+\nu_1+j]}^\infty
    =\sum_{k+\ell=j}a_{1,[\nu_1+\ell]}^\infty a_{0,[\nu_0+k]}^\infty.$$
\item[ii$)$] $D^\alpha_\xi D^\beta_x: 
	\wtbfS^{d,\nu}_{1,0}\to\wtbfS^{d-|\alpha|,\nu-|\alpha|}_{1,0}$ with $
	(D^\alpha_\xi D^\beta_x a)_{[\nu-|\alpha|]}^\infty=D^\alpha_\xi D^\beta_x a_{[\nu]}^\infty$,  
\item[iii$)$] $\partial_\mu^j: \wtbfS^{d,\nu}_{1,0}\to\wtbfS^{d-j,\nu}_{1,0}$ with 
	$(\partial^j_\mu a)_{[\nu]}^\infty=(d-\nu)(d-\nu-1)\ldots(d-\nu-j+1)a_{[\nu]}^\infty$. 
\end{itemize}
\end{lemma}
\begin{proof}
i) is straight-forward, as is ii) using induction on $|\alpha|$. 

By induction, it is enough to show iii) for $j=1$. Observe that    
 $$\partial_{\mu}[\xi,\mu]^{d-\nu-j}\equiv (d-\nu-j)[\xi,\mu]^{d-\nu-j-2}\mu
   \mod\scrC^\infty_{\mathrm{comp}}(\rpbar\times\rz^n).$$
Now use the expansion of $\mu\in \wtbfS^{1,0}_{1,0}$, cf. Proposition \ref{prop:classical}, 
to find a resulting expansion of $\partial_\mu a$.  
\forget{
i$)$ By induction, it suffices to consider the case $|\alpha|=1$. 
Now, using the notation of Definition \ref{def:symbol-expansion}, 
 $$D_{\xi_k} D^\beta_x a(x,\xi;\mu)
   =D_{\xi_k}\sum_{j=0}^{N-1}[\xi,\mu]^{d-\nu-j}D^\beta_xa_{[\nu+j]}^\infty(x,\xi)
   +D_{\xi_k} D^\beta_x r_{a,N}(x,\xi;\mu).$$   
Now observe that 
 $$D_{\xi_k}[\xi,\mu]^{d-\nu-j}\equiv -i(d-\nu-j)[\xi,\mu]^{d-\nu-j-2}\xi_k
   \mod\scrC^\infty_{\mathrm{comp}}(\rpbar\times\rz^n).$$
It follows that 
 $$D_{\xi_k} D^\beta_x a(x,\xi;\mu)\equiv    
   \sum_{j=0}^{N-1}[\xi,\mu]^{d-\nu-j}b_{[\nu-1+j]}^\infty(x,\xi) 
   \mod \wt{S}^{d-|\alpha|,\nu-|\alpha|+N}_{1,0},$$
where $b_{[\nu-1+j]}^\infty=D_{\xi_k} D^\beta_x a_{[\nu+j]}^\infty$ for $j=0$ and $j=1$, 
while 
 $$b_{[\nu-1+j]}^\infty=D_{\xi_k} D^\beta_x a_{[\nu+j]}^\infty
   -i(d-\nu-j+2)\xi_ka_{[\nu+j-2]}^\infty$$ 
for $j\ge 2$. Continuity follows again with the closed graph theorem. 

ii) One also proceeds by induction. Observe that    
 $$D_{\mu}[\xi,\mu]^{d-\nu-j}\equiv -i(d-\nu-j)[\xi,\mu]^{d-\nu-j-2}\mu
   \mod\scrC^\infty_{\mathrm{comp}}(\rpbar\times\rz^n).$$
Now use the expansion of $\mu\in \wtbfS^{1,0}_{1,0}$, cf. Proposition \ref{prop:classical}, 
to find a resulting expansion of $D_\mu a$. We leave the details to the reader. 
}
\end{proof}

\begin{theorem}[Asymptotic summation]\label{thm:asympsumm}
Let $a_j\in\wtbfS^{d-j,\nu-j}_{1,0}$, $j\in\nz_0$. 
Then there exists an $a\in\wtbfS^{d,\nu}_{1,0}$ such that 
$a-\sum\limits_{j=0}^{N-1}a_j\in \wtbfS^{d-N,\nu-N}_{1,0}$ for every $N$. Moreover, 
 $$a_{[\nu+j]}^\infty\sim\sum_{k=0}^{+\infty}a_{k,[(\nu-k)+j]}^\infty,\qquad j\in\nz_0,$$
asymptotically in $S^{\nu+j}_{1,0}(\rz^n)$. The symbol $a$ is unique 
modulo $\wtbfS^{d-\infty,\nu-\infty}_{1,0}$. 
\end{theorem}
\begin{proof}
Let $\chi(\xi)$ be a zero-excision function and denote by $\chi_c$, $c>0$, 
the operator of multiplication by $\chi(\xi/c)$. 
Then $\chi_c\in\scrL(\wtbfS^{d,\nu}_{1,0})$ for every $d,\nu$ with 
$(\chi_ca)_{[\nu+j]}^\infty=\chi_ca_{[\nu+j]}^\infty$ and 
$(1-\chi_c)a\in \wtbfS^{d-\infty,\nu-\infty}_{1,0}$. 

Moreover, the following statements are checked by straight-forward calculations$:$ 
\begin{itemize}
	\item[$(1)$] If $a\in \wt{S}^{d-1,\nu-1}_{1,0}$ then $\chi_ca\xrightarrow{c\to+\infty}0$ in 
	 $\wt{S}^{d,\nu}_{1,0}$. 
	\item[$(2)$] If $a\in \wtbfS^{d-1,\nu-1}_{1,0}$ then 
	 $\chi_ca_{[\nu-1+j]}^\infty\xrightarrow{c\to+\infty}0$ in ${S}^{\nu+j}_{1,0}(\rz^n)$. 
	\item[$(3)$] If $r\in \wt{S}^{d-1,\nu-1+N}_{1,0}$ then $\chi_cr\xrightarrow{c\to+\infty}0$ 
	 in $\wt{S}^{d,\nu+N}_{1,0}$.
\end{itemize}
In other words, given $a\in\wtbfS^{d-1,\nu-1}_{1,0}$ then 
$\chi_c a\xrightarrow{c\to+\infty}0$ in $\wtbfS^{d,\nu}_{1,0}$. 

Now the existence of $a$ follows from Proposition 1.1.17 of \cite{Schulze-Wiley} 
(with $E^j:=\wtbfS^{d-j,\nu-j}_{1,0}$ and $\chi^j(c)=\chi_c:E^j\to E^j$). 
The remaining statements are clear. 
\end{proof}

For the detailed proofs of the following two theorems, concerning composition and $($formal$)$ adjoint, 
see the appendix.  

\begin{theorem}\label{thm:Leibniz product}
Let $a_j\in \wtbfS^{d_j,\nu_j}_{1,0}$, $j=0,1$. 
Then $a_1\#a_0 \in \wtbfS^{d_0+d_1,\nu_0+\nu_1}_{1,0}$ 
and the limit-symbol behaves multiplicatively$:$ 
$(a_1\#a_0)_{[\nu_0+\nu_1]}^\infty=a_{1,[\nu_1]}^{\infty}\# a_{0,[\nu_0]}^{\infty}$. 
Moreover, 
\begin{equation}
  a_1\#a_0\equiv\sum_{|\alpha|=0}^{N-1}\frac{1}{\alpha!}(\partial^\alpha_\xi a_1)(D^\alpha_x a_0)
    \mod \wtbfS^{d_0+d_1-N,\nu_0+\nu_1-N}_{1,0}.
\end{equation}
\end{theorem}

\begin{theorem}\label{thm:adjoint}
If $a\in\wtbfS^{d,\nu}_{1,0}$ then $a^{(*)}\in\wtbfS^{d,\nu}_{1,0}$ with 
$(a^{(*)})_{[\nu]}^\infty=(a_{[\nu]}^{\infty})^{(*)}$ and  
\begin{equation}\label{eq:leibniz-expansion2}
 a^{(*)}(x,\xi;\mu)=\sum_{|\alpha|=0}^{N-1}\frac{1}{\alpha!}\partial^\alpha_\xi D^\alpha_x 
 \overline{a(x,\xi;\mu)}\mod \wtbfS^{d-N,\nu-N}_{1,0}.
\end{equation}
\end{theorem}

\subsection{Ellipticity and parametrix construction}
\label{sec:05.2}
For the following considerations it is convenient to introduce the spaces 
$S^d_{1,0}(\rpbar;E)$ consisting of all smooth functions $a(\mu)$ with values in a 
Fr\`{e}chet space $E$, such that 
 $$\trinorm{{D^j_\mu a(\mu)}}\lesssim \spk{\mu}^{d-j}$$ 
for every $j$ and every continuous semi-norm $\trinorm{\cdot}$ of $E$. 

\begin{lemma}\label{prop:1+r}
Let $a(x,\xi;\mu)\in\wtbfS^{0-\infty,0-\infty}_{1,0}$ and assume that $1-a_{[0]}^\infty(x,D)$ is 
invertible in $\scrL(H^s(\rz^n))$ for some $s\in\rz$. 
Then $1-a(x,D;\mu)$ is invertible for large $\mu$ and  
$(1-a(x,D;\mu))^{-1}=1-b(x,D;\mu)$ for some $b(x,\xi;\mu)\in\wtbfS^{0-\infty,0-\infty}_{1,0}$.
\end{lemma}
\begin{proof}
First observe that $\wtbfS^{0-\infty,0-\infty}_{1,0}$ consists of those symbols $a$ for 
which exists a sequence of symbols $a_{[j]}^\infty\in S^{-\infty}(\rz^n)$ such that, 
for every $N\in\nz_0$, 
 $$a(x,\xi;\mu)=\sum_{j=0}^{N-1}a_{[j]}^\infty(x,\xi)[\xi,\mu]^{-j} \mod 
   \wt{S}^{0-\infty,N-\infty}_{1,0}=S^{-N}_{1,0}(\rpbar,S^{-\infty}(\rz^n)).$$
Also recall that $(1-T)^{-1}=\sum_{j=0}^{N-1}T^j+T^N(1-T)^{-1}$ whenever 
$T$ belongs to a unital algebra and $1-T$ is invertible.  
 
\textbf{Step 1:} Let us assume that $a_{[0]}^\infty=0$. In particular, 
$a\in S^{-1}_{1,0}(\rpbar,S^{-\infty}(\rz^n))$. 

Due to the spectral-invariance of the algebra $\{p(x,D)\mid p\in S^{0}_{1,0}(\rz^n)\}$ 
in $\scrL(H^s(\rz^n))$, we find that $1-a$ is invertible with respect to the Leibniz product for 
large $\mu$ and that $\chi(\mu)(1-a(\mu))^{-\#}$ belongs to 
$S^{0}_{1,0}(\rpbar,S^{0}(\rz^n))$  for a suitable zero-excision function $\chi$.  
But then 
 $$b:=-a- a\#\chi(1-a)^{-\#}\#a\in S^{-1}_{1,0}(\rpbar,S^{-\infty}(\rz^n))$$
and $1-b=(1-a)^{-\#}$ for large $\mu$. Hence, for large $\mu$, 
 $$(1-a)^{-\#}=1+\sum_{j=1}^{N-1}a^{\#j}+a^{\#N}\#(1-b)\equiv 1+\sum_{j=1}^{N-1}a^{\#j} 
   \mod S^{-N}_{1,0}(\rpbar,S^{-\infty}(\rz^n)).$$
Using the expansions of $a^{\#j}\in \wtbfS^{0-\infty,0-\infty}_{1,0}$ and noting that 
$(a^{\#j})_{[0]}^\infty=0$ for every $j$ due to the multiplicativity of the principal limit-symbol, 
we find a sequence of symbols $b_{[j]}^\infty\in S^{-\infty}(\rz^n)$ such that, 
for every $N\in\nz_0$, 
 $$(1-a)^{-\#}=1+\sum_{j=1}^{N-1}b_{[j]}^\infty[\xi,\mu]^{-j}+ r_N,\qquad  
   r_N\in S^{-N}_{1,0}(\rpbar,S^{-\infty}(\rz^n)),$$
for large $\mu$. Thus, for a suitable zero-excision function $\kappa(\mu)$, 
\begin{align*}
 1-b=\kappa(1-b)+(1-\kappa)-(1-\kappa)b\equiv \kappa(1-a)^{-\#}+(1-\kappa)\mod S^{-\infty},
\end{align*}
hence 
\begin{align*}
 1-b\equiv 1+\kappa\sum_{j=1}^{N-1}b_{[j]}^\infty[\xi,\mu]^{-j}+ \kappa r_N
 \equiv 1+\sum_{j=1}^{N-1}b_{[j]}^\infty[\xi,\mu]^{-j},  
\end{align*}
modulo $S^{-N}_{1,0}(\rpbar,S^{-\infty}(\rz^n))$, since 
$(1-\kappa)b_{[j]}^\infty\in S^{-\infty}$. 

\textbf{Step 2:} In the general case, again by spectral invariance, we find a 
$b_{[0]}^\infty\in S^{-\infty}(\rz^n)$ such that $1-b_{[0]}^\infty(x,D)$ is the 
inverse of $1-a_{[0]}^\infty(x,D)$. Then $(1-a)\#(1-b_{[0]}^\infty)=1-a^\prime$, 
where $a^\prime\in\wtbfS^{0-\infty,0-\infty}_{1,0}$ has vanishing principal limit-symbol. 
Apply Step 1 to $1-a^\prime$ to find a corresponding parametrix $1-b^\prime$. 
Then the claim follows by choosing 
$b=1-(1-b_{[0]}^\infty)\#(1-b^\prime)=b^\prime+b_{[0]}^\infty-b_{[0]}^\infty\#b^\prime$. 
\end{proof}

\begin{definition}\label{def:ell1}
We call $a\in\wtbfS^{d,\nu}_{1,0}$ elliptic if there exist an $R\ge0$ such that 
\begin{itemize} 
 \item[$(1)$] $a(x,\xi;\mu)$ is invertible whenever $|\xi|\ge R$ and   
  $$|a(x,\xi;\mu)^{-1}|\lesssim\spk{\xi}^{-\nu}\spk{\xi,\mu}^{\nu-d},$$ 
 \item[$(2)$] $a_{[\nu]}^\infty(x,D)$ is invertible in $\scrL(H^s(\rz^n), H^{s-\nu}(\rz^n))$ 
  for some $s\in\rz$.
\end{itemize}
\end{definition}

Note that condition $(2)$ is equivalent to the existence of an inverse of $a_{[\nu]}^\infty(x,\xi)$ 
with respect to the Leibniz product, with inverse belonging to $S^{-\nu}_{1,0}(\rz^n)$. 

\begin{theorem}\label{thm:param_exp}
Let $a\in\wtbfS^{d,\nu}_{1,0}$ be elliptic. Then there exists a $b\in\wtbfS^{-d,-\nu}_{1,0}$ 
such that both $a\#b-1$ and $b\#a-1$ belong to 
$\scrC^\infty_{\mathrm{comp}}(\rpbar,S^{-\infty}(\rz^n))$. 
In particular, $b(x,D;\mu)=a(x,D;\mu)^{-1}$ provided $\mu$ is sufficiently large. 
\end{theorem}
\begin{proof}
By order reduction we may assume without loss of generality that $d=\nu=0$. 

\textbf{Step 1:} Since $a_{[0]}^\infty(x,D)$ is invertible, $a_{[0]}^\infty$ is also elliptic. 
Thus we can choose a zero-excision function $\chi(\xi)$ such that 
\begin{align*} 
 \chi(2\xi)a(x,\xi;\mu)^{-1}\in \wt{S}^{0,0}_{1,0},\qquad  
 \wt{c}_{[0]}^\infty(x,\xi):=\chi(2\xi)a_{[0]}^\infty(x,\xi)^{-1}\in S^0_{1,0}(\rz^n),
\end{align*} 
and $\chi(\xi)\chi(2\xi)=\chi(\xi)$. Now define recursively, 
 $$\wt{c}_{[j]}^\infty(x,\xi)=-\wt{c}_{[0]}^\infty(x,\xi)\sum_{\substack{k+\ell=j,\\ \ell<j}}
   a_{[k]}^\infty(x,\xi)\wt{c}_{[\ell]}^\infty(x,\xi),\qquad j\in\nz,$$
and set ${c}_{[j]}^\infty(x,\xi)=\chi(\xi)\wt{c}_{[j]}^\infty(x,\xi)$. Then 
 $$\Big(\sum_{j=0}^{N-1}{a}_{[j]}^\infty(x,\xi)[\xi,\mu]^{-j}\Big)
   \Big(\sum_{j=0}^{N-1}{c}_{[j]}^\infty(x,\xi)[\xi,\mu]^{-j}\Big)= \chi(\xi) - r_N(x,\xi;\mu)$$  
with $r_N\in\wt{S}^{0,N}_{1,0}$. 
Thus, if $c(x,\xi;\mu):=\chi(\xi)a(x,\xi;\mu)^{-1}$, then 
\begin{align*} 
 c(x,\xi;\mu)&=\chi(\xi)\chi(2\xi)a(x,\xi;\mu)^{-1}\\
 &\equiv \Big(\sum_{j=0}^{N-1}{a}_{[j]}^\infty(x,\xi)[\xi,\mu]^{-j}\Big)
   \Big(\sum_{j=0}^{N-1}{c}_{[j]}^\infty(x,\xi)[\xi,\mu]^{-j}\Big)\chi(2\xi)a(x,\xi;\mu)^{-1}
\end{align*} 
modulo $\wt{S}^{0,N}_{1,0}$. The first factor on the right-hand side equals $a-r_{a,N}$ with 
$r_{a,N}\in\wt{S}^{0,N}$. It follows that 
\begin{align*} 
 c(x,\xi;\mu)\equiv \sum_{j=0}^{N-1}{c}_{[j]}^\infty(x,\xi)[\xi,\mu]^{-j}\mod \wt{S}^{0,N}_{1,0}. 
\end{align*} 
This shows that $c(x,\xi;\mu)=\chi(\xi)a(x,\xi;\mu)^{-1}\in  \wtbfS^{0,0}_{1,0}$. 

\textbf{Step 2:} Let $c$ as constructed in Step 1. Then, by Theorem \ref{thm:Leibniz product}, 
$a\#c\equiv ac=\chi(\xi)$ modulo $\wtbfS^{-1,-1}_{1,0}$. Thus $a\#c-1\in \wtbfS^{-1,-1}_{1,0}$ 
and the usual Neumann series argument, which is possible in view of Theorem \ref{thm:asympsumm},
allows to construct a symbol $c^\prime\in\wtbfS^{0,0}_{1,0}$ such that $a\#c^\prime=1-r$ with 
$r\in\wtbfS^{0-\infty,0-\infty}_{1,0}$. Now define 
 $$c^{\prime\prime}:=c^\prime+(a_{[0]}^\infty)^{-\#}\# r_{[0]}^{\infty};$$
note that $r_{[0]}^{\infty}\in S^{-\infty}(\rz^n)$, hence 
$c^{\prime\prime}-c^\prime\in\wtbfS^{0-\infty,0-\infty}_{1,0}$. It follows that 
$a\#c^{\prime\prime}-1\in \wtbfS^{0-\infty,0-\infty}_{1,0}$ and 
 $$(a\#c^{\prime\prime}-1)_{[0]}^\infty
     =a_{[0]}^\infty\#((c^\prime)_{[0]}^\infty+(a_{[0]}^\infty)^{-1}\#r_{[0]}^\infty)-1
     =a_{[0]}^\infty\#(c^\prime)_{[0]}^\infty+r_{[0]}^\infty-1=0$$ 
by construction. Thus $a\#c^{\prime\prime}=1-r^\prime$, where 
$r^\prime\in\wtbfS^{0-\infty,0-\infty}_{1,0}$ has vanishing limit-symbol. 
Using Proposition \ref{prop:1+r} we thus find a right-parametrix $b_R\in\wtbfS^{0,0}_{1,0}$ such that 
$a\#b_R-1\in\scrC^\infty_{\mathrm{comp}}(\rpbar,S^{-\infty}(\rz^n))$. 

Analogously, we construct a left-parametrix $b_L$. Then the claim follows by choosing 
$b=b_L$ or $b=b_R$. 
\end{proof}

\subsection{Poly-homogeneous symbols with expansion at infinity}
\label{sec:05.3}

As already mentioned $\wt{S}^{d,\nu}_{hom}\cong r^\nu\scrC^\infty_B(\whsz^n_+)$ does not  
behave well under inversion because there is no sufficient control at the singularity. We pass 
to a subclass which also is compatible with the previously introduced expansion at infinity. 

\begin{definition}\label{def:taylor}
A function $\wh{a}(\xi;\mu)\in r^\nu\scrC^\infty_B(\wh{\sz}^n_+)$ is said to have a weighted 
Taylor expansion in $(0,1)$, if there exist $\wh{a}_{\spk{\nu+j}}\in\scrC^\infty(\sz^{n-1})$, 
$j\in\nz_0$, such that the representation $\wh{a}(r,\phi)=\wh{a}(r\phi;\sqrt{1-r^2})$ of $a$ in 
polar-coordinates satisfies 
 $$\wh{a}(r,\phi)-\omega(r)\sum_{j=0}^{N-1}r^{\nu+j}\wh{a}_{\spk{\nu+j}}(\phi)\in 
    r^{\nu+N}\scrC^\infty_B((0,1)\times\sz^{n-1})$$
for every $N\in\nz_0$, where $\omega\in\scrC^\infty([0,1))$ is a cut-off function. 
The space of all such functions will be denoted by $r^\nu\scrC^\infty_T(\wh{\sz}^n_+)$. 
\end{definition}

Note that $\wh{a}(\xi;\mu)\in r^\nu\scrC^\infty_T(\wh{\sz}^n_+)$ is invertible with inverse in 
$r^{-\nu}\scrC^\infty_T(\wh{\sz}^n_+)$ if, and only if, $\wh{a}(\xi;\mu)\not=0$ whenever 
$\xi\not=0$ and $\wh{a}_{\spk{\nu}}(\phi)\not=0$ for all $\phi\in\sz^{n-1}$. 

\begin{definition}\label{def:angular}
$\wtbfS^{d,\nu}_{hom}$ consists of all functions of the form 
 $$a(x,\xi;\mu)=|\xi,\mu|^d \,\wh{a}\Big(x,\frac{(\xi,\mu)}{|\xi,\mu|}\Big),\qquad 
   \wh{a}\in \scrC^\infty_b\big(\rz^n_x,r^\nu\scrC^{\infty}_{T}(\whsz^n_+)\big).$$
Define the \textit{principal angular symbol} 
$a_{\spk{\nu}}(x,\xi)\in S^{\nu}_{hom}(\rz^n)$ of $a$ as  
  $$a_{\spk{\nu}}(x,\xi)=|\xi|^{\nu}\,\wh{a}_{\spk{\nu}}\Big(x,\frac{\xi}{|\xi|}\Big)
      =|\xi|^{\nu}\lim_{r\to 0+}r^{-\nu}\,a\Big(x,r\frac{\xi}{|\xi|};\sqrt{1-r^2}\Big).$$
\end{definition}

Note that, by construction, $\wtbfS^{d,\nu}_{hom}\subseteq\wt{S}^{d,\nu}_{hom}$. The following 
proposition shows that such homogeneous components intrinsically admit an expansion at infinity 
in the sense of Definition \ref{def:symbol-expansion}. 

\begin{proposition}\label{prop:excision}
Let $a(x,\xi;\mu)\in \wtbfS^{d,\nu}_{hom}$ be as in Definition $\ref{def:angular}$ with $\wh{a}$ 
as in Definition $\ref{def:taylor}$. Let $p(x,\xi;\mu)=\chi(\xi)a(x,\xi;\mu)$ with a zero-excision function 
$\chi(\xi)$. Then $p\in\wtbfS^{d,\nu}_{1,0}$ with  
\begin{align*}
 p_{[\nu+j]}^\infty(x,\xi;\mu)=\chi(\xi)|\xi|^{\nu+j}{a}_{\spk{\nu+j}}\Big(x,\frac{\xi}{|\xi|}\Big),
\qquad j\in\nz_0. 
\end{align*}
\end{proposition}
\begin{proof}
By Theorem \ref{thm:main01},  
\begin{align*}
 p(x,\xi;\mu)
 &\equiv \chi(\xi)|\xi,\mu|^d\omega\Big(\frac{|\xi|}{|\xi,\mu|}\Big)
   \sum_{j=0}^{N-1}\Big(\frac{|\xi|}{|\xi,\mu|}\Big)^{\nu+j}
   {a}_{\spk{\nu+j}}\Big(x,\frac{\xi}{|\xi|}\Big)\\
 &\equiv \chi(\xi)\omega\Big(\frac{|\xi|}{|\xi,\mu|}\Big)
  \sum_{j=0}^{N-1}|\xi|^{\nu+j}{a}_{\spk{\nu+j}}\Big(x,\frac{\xi}{|\xi|}\Big) [\xi,\mu]^{d-\nu-j}
\end{align*}
modulo $\wt{S}^{d,\nu+N}$ for every $N$. 
Now observe that $w(\xi;\mu):=(1-\omega)\Big(\frac{|\xi|}{|\xi,\mu|}\Big)$ is a smooth function 
on $(\rz^n\times\rpbar)\setminus\{0\}$ which is homogeneous of degree $0$ and is supported in 
a set of the form $\{(\xi,\mu)\mid 0\le\mu\le c|\xi|\}$. 
Thus, if $\kappa(\xi,\mu)$ is a zero-excision function, then 
$\kappa(\xi,\mu)w(\xi;\mu)\in \wt{S}^{0,L}_{1,0}$ for every $L$, since on its support 
$\spk{\xi,\mu}\sim\spk{\xi}$. Choosing $\kappa$ such that $\kappa(\xi,\mu)\chi(\xi)=\chi(\xi)$, 
we conclude that  
\begin{align*}
p(\xi;\mu)
\equiv \sum_{j=0}^{N-1}\chi(\xi)|\xi|^{\nu+j}{a}_{\spk{\nu+j}}\Big(x,\frac{\xi}{|\xi|}\Big)
[\xi,\mu]^{d-\nu-j}  
\end{align*}
modulo $\wt{S}^{d,\nu+N}$ for every $N$. This concludes the proof. 
\end{proof}

\begin{definition}\label{def:wtbfs-poly}
The space $\wtbfS^{d,\nu}$ consists of all symbols 
$a(x,\xi;\mu)\in \wtbfS^{d,\nu}_{1,0}$ for which exists a sequence of homogeneous components 
$a_j(x,\xi;\mu)\in \wtbfS^{d-j,\nu-j}_{hom}$ such that 
\begin{align}\label{eq:asymptilde}
 a(x,\xi;\mu)-\chi(\xi)\sum_{j=0}^{N-1}a_j(x,\xi;\mu)\in \wtbfS^{d-N,\nu-N}_{1,0}
\end{align}
for every $N\in\nz_0$. 
The \emph{principal angular symbol} of $a(x,\xi;\mu)$ is, by definition, 
the principal angular symbol $a_{0,\spk{\nu}}(x,\xi)$ of the homogeneous principal symbol of 
$a_0(x,\xi;\mu)$ $($cf. Definition $\ref{def:angular})$. 
\end{definition}

The previous definition is meaningful according to Proposition \ref{prop:excision} and 
Theorem \ref{thm:asympsumm}. 
If $a$ is as in \eqref{eq:asymptilde} then the principal limit-symbol $a_{[\nu]}^\infty(x,\xi)$ belongs to 
$S^{\nu}(\rz^n)$ and has the asymptotic expansion 
 $$a_{[\nu]}^\infty(x,\xi)\sim\chi(\xi)\sum_{j=0}^{+\infty}|\xi|^{\nu-j}
   a_{j,\spk{\nu-j}}\Big(x,\frac{\xi}{|\xi|}\Big).$$
In particular, we have the following$:$

\begin{proposition}\label{prop:compatibility}
Let $a(x,\xi;\mu)\in\wtbfS^{d,\nu}$. Then the homogeneous principal symbol of the principal 
limit-symbol $a_{[\nu]}^\infty(x,\xi)$ coincides with the principal angular symbol of $a(x,\xi;\mu)$. 
\end{proposition}

Now let us turn to ellipticity and parametrix. 

\begin{definition}\label{def:ell}
A symbol $a(x,\xi;\mu)\in\wtbfS^{d,\nu}$ is called elliptic if  
\begin{itemize} 
 \item[$(1)$] The homogeneous principal symbol $a_0(x,\xi;\mu)$ is invertible whenever $\xi\not=0$ and   
  $$|a_0(x,\xi;\mu)^{-1}|\lesssim|\xi|^{-\nu}|\xi,\mu|^{\nu-d}.$$ 
 \item[$(2)$] $a_{[\nu]}^\infty(x,D)$ is invertible in $\scrL(H^s(\rz^n), H^{s-\nu}(\rz^n))$ for  
  some $s\in\rz$. 
\end{itemize}
\end{definition}

Due to Proposition \ref{prop:compatibility}, condition $(2)$ implies the invertibility of the 
principal angular symbol of $a$. 
Moreover, if the homogeneous principal symbol of $a(x,\xi;\mu)$ does not depend on $x$ 
for large $|x|$, then condition $(1)$ in Definition \ref{def:ell} can be substituted by  
\begin{itemize} 
 \item[$(1^\prime)$] The homogeneous principal symbol $a_0(x,\xi;\mu)$ 
 is invertible whenever $\xi\not=0$. 
\end{itemize}

\begin{theorem}\label{thm:param}
Let $a\in\wtbfS^{d,\nu}$ be elliptic. Then there exists a $b\in\wtbfS^{-d,-\nu}$ such that both 
$a\#b-1$ and $b\#a-1$ belong to 
$\scrC^\infty_{\mathrm{comp}}(\rpbar,S^{-\infty}(\rz^n))$. 
In particular, $b(x,D;\mu)=a(x,D;\mu)^{-1}$ provided $\mu$ is sufficiently large. 
\end{theorem}
\begin{proof}
By ellipticity assumption $(2)$, there exists a $b(x,\xi)\in S^{-\nu}(\rz^n)$ 
which is the inverse of $a_{[\nu]}^\infty(x,\xi)$ with respect to the Leibniz product. 
By Proposition \ref{prop:compatibility} it follows that the principal angular symbol of $a$ 
$($i.e., that of $a_0)$ is invertible and the inverse is just the homogeneous principal 
symbol of $b$. Together with $(1)$ we conclude that the homogeneous principal symbol 
$a_0(x,\xi;\mu)$ is invertible with inverse belonging to $\wtbfS^{-d,-\nu}_{hom}$. 
Thus there exists a $c(x,\xi;\mu)\in \wtbfS^{-d,-\nu}$ which is a parametrix of 
$a(x,\xi;\mu)$ modulo $\wtbfS^{-1,-1}$. Then proceed as in Step $2$ of the proof 
of Theorem \ref{thm:param_exp}. 
\end{proof}

\subsection{Refined calculus for symbols of finite regularity}
\label{sec:05.4}

As proved in Theorem \ref{thm:decomposition}, Grubb's class $S^{d,\nu}$ coincides with the 
non-direct sum $\wt{S}^{d,\nu}+S^d$. 
In light of the above considerations it is now natural to introduce the following class: 

\begin{definition}\label{def:new_reg}
With $d\in\rz$ and $\nu\in\gz$ define 
 $$\bfS^{d,\nu}= \wtbfS^{d,\nu}+S^d,\qquad 
     \bfS^{d,\nu}_{hom}= \wtbfS^{d,\nu}_{hom}+S^d_{hom}.$$ 
\end{definition}

The limitation to integer values of $\nu$ is due to Lemma \ref{lem:inv}, below. 
Note also that $\bfS^{d,\nu}=\wtbfS^{d,\nu}$ whenever $\nu\le0$, and  
$\bfS^{d,\nu}\subset\wtbfS^{d,0}$ whenever $\nu>0$. 

By Proposition \ref{prop:classical} and Theorem \ref{thm:Leibniz product}, 
the Leibniz product induces maps 
 $$\bfS^{d_1,\nu_1}\times\bfS^{d_0,\nu_0}\lra \bfS^{d_0+d_1,\nu},\qquad 
   \nu=\min(\nu_0,\nu_1,\nu_0+\nu_1).$$
By Theorem \ref{thm:adjoint} the class is closed under taking the $($formal$)$ adjoint. 
Since in both spaces involved in Definition \ref{def:new_reg} asymptotic summation is 
possible $($cf. Section \ref{sec:03.2} and Theorem \ref{thm:asympsumm}$)$, a 
sequence of symbols $a_j\in \bfS^{d-j,\nu-j}$ can be summed asymptotically in $\bfS^{d,\nu}$. 

\begin{lemma}\label{lem:inv}
Let $\nu\in\nz$ be a positive integer and 
$a(\xi;\mu)\in r^\nu\scrC^\infty_T(\wh\sz^n_+)+\scrC^\infty(\sz^n_+)$ 
be point-wise invertible on $\sz^n_+$.\footnote{Recall that $a(\xi;\mu)$ extends by continuity 
to the whole semi-sphere, since $\nu>0$.} 
Then $a(\xi;\mu)^{-1}\in r^\nu\scrC^\infty_T(\wh\sz^n_+)+\scrC^\infty(\sz^n_+)$. 
\end{lemma}
\begin{proof}
Write $a=\wt{a}+a_0$ with $\wt{a}\in r^\nu\scrC^\infty_T(\wh\sz^n_+)$ and 
$a_0\in \scrC^\infty(\sz^n_+)$. Clearly $a$ is smooth on $\wh\sz^n_+$. 
Since $a_0(0;1)=a(0;1)$ is invertible, there exists a $b_0\in\scrC^\infty(\sz^n_+)$ everywhere 
invertible and such that $b_0=a_0$ in a neighborhood of $(0,1)$. 
Then $ab_0^{-1}$ has the same structure as $a$, hence we may assume from the very 
beginning, that $a_0\equiv1$ in a neighborhood $(0,1)$. 
Now let $\psi_1,\psi\in\scrC^\infty(\sz^n_+)$ be supported in this neighborhood, 
such that $\psi,\psi_1\equiv1$ near $(0,1)$ and $\psi\psi_1=\psi$. 
Then $\psi a^{-1}=\psi(1+\psi_1\wt{a})^{-1}$. Taking the support of $\psi_1$ and $\psi$ 
sufficiently small, we have that $\wt{b}:=-(\psi_1\wt{a})\in r^\nu\scrC^\infty_T(\wh\sz^n_+)$ 
and $|\wt{b}|\le 1/2$ on $\sz^n_+$. By chain rule it is straight-forward to see that 
$(1-\wt{b})^{-1}\in\scrC^\infty_B(\wh\sz^n_+)$. Then 
 $$(1-\wt{b})^{-1}=1+\sum_{j=1}^{N-1}\wt{b}^j+\wt{b}^N (1-\wt{b})^{-1},\qquad N\in\nz,$$
shows that $(1-\wt{b})^{-1}=1-\wt{c}$ with $\wt{c}\in r^\nu\scrC^\infty_T(\wh\sz^n_+)$. 
This yields the claim. 
\end{proof}

In case of $x$-dependence we need to pose, as usual, an additional uniform bound on the inverse. 
Since symbols of $\bfS^{d,\nu}_{hom}$ are just the homogeneous extensions of degree $d$ of 
functions from $r^\nu\scrC^\infty_T(\wh\sz^n_+)+\scrC^\infty(\sz^n_+)$, we immediately have 
the following corollary. 

\begin{corollary}\label{cor:inv}
Let $\nu\in\nz$ be positive and $a(x,\xi;\mu)\in \bfS^{d,\nu}_{hom}$.  
Assume that $a(x,\xi;\mu)$ is invertible whenever $(\xi,\mu)\not=0$ 
and that $|a(x,\xi;\mu)^{-1}|\lesssim |\xi,\mu|^{-d}$. 
Then $a(x,\xi;\mu)^{-1}\in\bfS^{-d,\nu}_{hom}$. 
\end{corollary}

After this observation it is clear that we can construct a parametrix in the class$:$

\begin{theorem}\label{thm:param02}
Let $\nu\in\nz$ be a positive integer and $a(x,\xi;\mu)\in \bfS^{d,\nu}$ be elliptic 
$($i.e., the homogeneous principal symbol satisfies the assumptions of Corollary $\ref{cor:inv})$. 
Then there exists a parametrix  $b(x,\xi;\mu)\in \bfS^{-d,\nu}$ such that both $a\#b-1$ and 
$b\#a-1$ belong to $\scrC^\infty_{\mathrm{comp}}(\rpbar,S^{-\infty}(\rz^n))$.   
\end{theorem}

If $a\in \bfS^{d,\nu}$ with positive $\nu\in\nz$, then also $a\in\wtbfS^{d,0}$. 
Due to Propositions \ref{prop:excision} and \ref{prop:classical}, its principal limit-symbol is 
 $$a_{[0]}^\infty(x,\xi)=a_0(x,0;1),$$
where $a_0$ is the homogeneous principal symbol of $a$ 
$($defined on $\sz^n_+$ by continuous extension$)$. 
Recalling Definition \ref{def:ell}, we find the following result which unifies the notions of 
ellipticity for symbols of regularity number $\nu=0$ and $\nu\in\nz$, respectively. 

\begin{proposition}\label{prop:uniform}
Let $\nu\in\nz$ be a non-negative integer and $a\in \bfS^{d,\nu}$. 
Then $a$ is elliptic if, and only if, $a$ is elliptic as an element of $\wtbfS^{d,0}$. 
\end{proposition}

\section{Resolvent-kernel expansions}
\label{sec:06}

We shall discuss  how our calculus allows to recover the well-known resolvent trace 
expansion for elliptic $\psi$do due to Grubb-Seeley, cf. \cite{GrSe}. 

In the following we shall write $r(x,\xi;\mu)=O(\mu^m,S^{M}_{1,0})$ 
if $\mu^{-m} r(\mu)\in S^{M}_{1,0}(\rz^n)$ uniformly in $\mu>0$. 

\subsection{Preparation}
\label{sec:06.1}

The following Lemma is a slight modification of  \cite[Lemma 1.3]{GrSe}. 

\begin{lemma}\label{lem:mu-expansion}
Let $a(x,\xi;\mu)\in S^m$ be homogeneous of degree $m$ for $|\xi,\mu|\ge1$. 
Let $m_+=\max(m,0)$. 
Then there exist symbols $\zeta_{j}(x,\xi)=\sum\limits_{|\alpha|=j}c_{j \alpha}(x)\xi^\alpha$ 
such that 
 $$a(x,\xi;\mu)=\sum_{j=0}^{N-1}\zeta_{j}(x,\xi)\mu^{m-j}
     +O(\mu^{m-N},S^{m_++N}_{1,0})$$
for every $N\in\nz$. In particular, $\zeta_0(x,\xi)=a(x,0;1)$ and 
$\mu^{-m}a(x,\xi;\mu)\to a(x,0;1)$ in 
$S^{m_++1}_{1,0}(\rz^n)$ as $\mu\to+\infty$. 
\end{lemma}
\begin{proof}
For convenience of notation assume independence of $x$. 
Obviously it suffices to consider $\mu\ge1$. 
Then $a(\xi;\mu)=\mu^ma(\xi/\mu;1)$. 
Let $u(t,\xi)=a(t\xi;1)$, $0\le t\le 1$. 
The $j$-th $t$-derivative of $u$ is 
 $$u^{(j)}(t,\xi)=\sum\limits_{|\alpha|=j}c_{\alpha}\xi^\alpha(\partial_\xi^\alpha a)(t\xi;1)$$
with certain universal constants $c_\alpha$. Thus the Taylor expansion of $u$ in $t$ centered in 
$t=0$ is of the form 
 $$u(t,\xi)=\sum_{j=0}^{N-1}\zeta_j(\xi)t^j+t^N\int_0^1 (1-\tau)^Nu^{(N)}(t\tau,\xi)\,d\tau$$
with polynomials $\zeta_j(\xi)$ as described. Then using the fact that 
 $$|\partial^\beta_\xi[(\partial_\xi^\alpha a)(t\tau\xi;1)]|\lesssim 
     (t\tau)^{|\beta|}\spk{t\tau\xi}^{m-|\alpha|-|\beta|}\lesssim 
     (t\tau)^{|\beta|}\spk{t\tau\xi}^{-|\beta|}\spk{\xi}^{m_+}
     \lesssim\spk{\xi}^{m_+-|\beta|},$$
for $0\le t,\tau\le 1$, the above integral belongs to $S^{m_++N}_{1,0}(\rz^n)$ 
uniformly in $0\le t\le 1$. Substituting $t=1/\mu$ yields the claim. 
\end{proof}

\forget{
\begin{lemma}\label{lem:mu-expansion}
Let $m\le 0$ be a real number. Then there exist homogeneous polynomials 
$\zeta_{m,j}(\xi)=\sum\limits_{|\alpha|=j}c_{m j \alpha}\xi^\alpha$ 
such that, for every $N\in\nz$, 
 $$[\xi,\mu]^m=\sum_{j=0}^{N-1}\zeta_{m,j}(\xi)\mu^{m-j}
     +O(\mu^{m-N},S^N_{1,0}).$$
\end{lemma}
\begin{proof}
Obviously it suffices to consider $\mu\ge1$. 
Then $[\xi,\mu]^m=|\xi,\mu|^m=\mu^m\spk{\xi/\mu}^m$. 
Let $p(\xi)=\spk{\xi}^m$ and define $u(t,\xi)=\spk{t\xi}^m$, $0\le t\le 1$. 
By direct computation, the $j$-th $t$-derivative of $u$ is 
 $$u^{(j)}(t,\xi)=\sum\limits_{|\alpha|=j}c_{\alpha}\xi^\alpha(\partial_\xi^\alpha p)(t\xi)$$
with certain universal constants $c_\alpha$. Thus the Taylor expansion of $u$ in $t$ centered in 
$t=0$ is of the form 
 $$u(t,\xi)=\sum_{j=0}^{N-1}\zeta_j(\xi)t^j+t^N\int_0^1 (1-\tau)^Nu^{(N)}(t\tau,\xi)\,d\tau$$
with polynomials $\zeta_j(\xi)$ as described. Then using the fact that 
 $$|\partial^\beta_\xi[(\partial_\xi^\alpha p)(t\tau\xi)]|\prec 
     (t\tau)^{|\beta|}\spk{t\tau\xi}^{m-|\alpha|-|\beta|}\prec 
     (t\tau)^{|\beta|}\spk{t\tau\xi}^{-|\beta|}
     \prec\spk{\xi}^{-|\beta|}$$
uniformly in $0\le t,\tau\le 1$, the above integral belongs to $S^{N}(\rz^n)$ uniformly in $0\le t\le 1$.  
Substituting $t=1/\mu$ the claim follows. 
\end{proof}
}

A case of particular interest below is that 
\begin{equation}\label{eq:bracket}
 [\xi,\mu]^m=\sum_{j=0}^{N-1}\zeta_{m,j}(\xi)\mu^{m-j}
     +O(\mu^{m-N},S^N_{1,0}) 
\end{equation}
whenever $m\le 0$; any $\zeta_{m,j}(\xi)$ is a homogeneous polynomial of degree $j$.  

\begin{corollary}\label{cor:mu-expansion}
Let $a(x,\xi;\mu)\in\wtbfS^{d,\nu}_{1,0}$ with $d-\nu\le0$ have the expansion 
 $$a(x,\xi;\mu)=\sum_{j=0}^{N-1} a^\infty_{[\nu+j]}(x,\xi)[\xi,\mu]^{d-\nu-j}
    \mod \wt{S}^{d,\nu+N}_{1,0}.$$ 
Then 
 $$a(x,\xi;\mu)=\sum_{\ell=0}^{N-1} q_{\ell}(x,\xi)\mu^{d-\nu-\ell}
    +O(\mu^{d-\nu-N},S^{\nu+N}_{1,0}),$$
where, with notation of $\eqref{eq:bracket}$, 
 $$q_{\ell}(x,\xi)=\sum_{j+k=\ell}a^\infty_{[\nu+j]}(x,\xi)\zeta_{d-\nu-j,k}(\xi)
 \in S^{\nu+\ell}_{1,0}(\rz^n).$$
\end{corollary}
\begin{proof}
First note that for $r(x,\xi;\mu)\in \wt{S}^{d,\nu+N}$, 
 $$ |\partial^\alpha_\xi\partial^\beta_x r(x,\xi;\mu)|\lesssim 
     \spk{\xi}^{\nu+N-|\alpha|}\spk{\xi,\mu}^{d-\nu-N},$$
hence $r(\mu)= O(\mu^{d-\nu-N},S^{\nu+N})$. 
Inserting the expansions 
 $$[\xi,\mu]^{d-\nu-j}=\sum_{k=0}^{N-j-1}\zeta_{d-\nu-j,k}(\xi)\mu^{d-\nu-j-k}
     +O(\mu^{d-\nu-N},S^{N-j}_{1,0})$$
the result follows immediately. 
\end{proof}

\begin{theorem}\label{thm:trace-expansion-01}
Let $a(x,\xi;\mu)\in\wtbfS^{d,\nu}$ with $d<-n$ and $d-\nu\le0$. Let 
 $$k(x,y;\mu)=\int e^{i(x-y)\xi}a(x,\xi;\mu)\,\dbar\xi$$
the distributional kernel of $a(x,D;\mu)$. 
Then there exist functions $c_\ell(x),c_\ell^\prime(x),c_\ell^{\prime\prime}(x)$, $j\in\nz_0$, which are 
continuous and bounded such that, for $\mu\to+\infty$,  
 $$k(x,x;\mu)\sim \sum_{j=0}^{+\infty}c_j(x)\mu^{d-j+n}+\sum_{\ell=0}^{+\infty}
    \big(c_\ell^\prime(x)\log\mu+c_\ell^{\prime\prime}(x)\big)\mu^{d-\nu-\ell}.$$
\end{theorem}
\begin{proof}
We follow closely the proof of \cite[Theorem 2.1]{GrSe}. 
Let $N$ be fixed. Choose, and fix, a $J\in\nz$ so large that 
\begin{align}\label{eq:order}
 \nu-J+N<-n
\end{align}
and write 
 $$a(x,\xi;\mu)=\chi(\xi)\sum_{j=0}^{J-1} a_j(x,\xi;\mu) + r(x,\xi;\mu),
    \qquad r\in \wtbfS^{d-J,\nu-J},$$
where $a_j(x,\xi;\mu)\in \wtbfS^{d-j,\nu-j}_{hom}$ and $\chi$ is a zero-excision function 
such that $1-\chi$ is supported in the unit-ball centered in the origin. 

By Corollary \ref{cor:mu-expansion} $($with $d,\nu$ replaced by $d-J,\nu-J)$ we have 
 $$r(x,\xi;\mu)=\sum_{\ell=0}^{N-1} q_{\ell}(x,\xi)\mu^{d-\nu-\ell}
    +O(\mu^{d-\nu-N},S^{\nu-J+N}_{1,0})$$
with $ q_{\ell}(x,\xi)\in S^{\nu-J+\ell}_{1,0}(\rz^n)$. Recalling \eqref{eq:order}, 
the associated kernel $k_r(x,y;\mu)$ satisfies 
 $$k_r(x,x;\mu)=\sum_{\ell=0}^{N-1} c^{\prime\prime}_{r,\ell}(x)
     \mu^{d-\nu-\ell}+O(\mu^{d-\nu-N}).$$
Now let $k_j(x,y;\mu)$ denote the kernel associated with $\chi(\xi)a_j(x,\xi;\mu)$. Decompose 
$k_j(x,x;\mu)=k_j^{(1)}(x,x;\mu)+k_j^{(2)}(x,x;\mu)$ with 
 $$k_j^{(1)}(x,x;\mu)=\int_{|\xi|\ge\mu}\chi(\xi)a_j(x,\xi;\mu)\,\dbar\xi.$$
Then, for every $\mu\ge 1$, using the homogeneity of $a_j$,  
 $$k_j^{(1)}(x,x;\mu)=\mu^{d-j+n}\int_{|\xi|\ge1}a_j(x,\xi;1)\,\dbar\xi=c_j(x)\mu^{d-j+n};$$
note that the integrand is bounded by $\spk{\xi}^{d-j}$, hence integrable since $d<-n$.  

Next choose $L$ with $L\ge N$ and $L>J-1-n-\nu$ $($i.e. $\nu-j+L>-n$ for every $j=0,\ldots,J-1)$. 
Apply Corollary \ref{cor:mu-expansion} $($with $d,\nu$ replaced by $d-j,\nu-j)$ to write 
\begin{align}\label{eq:kernel01}
\begin{split}
 \chi(\xi)a_j(x,\xi;\mu)&=\sum_{\ell=0}^{L-1} q_{j,\ell}(x,\xi)\mu^{d-\nu-\ell}+s_{j,L}(x,\xi;\mu),\\
 s_{j,L}(x,\xi;\mu)&=O(\mu^{d-\nu-L},S^{\nu-j+L}_{1,0});
\end{split}
\end{align}
by Proposition 5.19 $($more precisely, the last formula in its proof$)$ the symbols 
$q_{j,\ell}(x,\xi)\in S^{\nu-j+\ell}_{1,0}(\rz^n)$ are homogeneous of degree 
$\nu-j+\ell$ for $|\xi|\ge1$. Thus $s_{j,L}(x,\xi;\mu)$ is homogeneous of degree $d-j$ in $(\xi,\mu)$ 
for $|\xi|\ge 1$. We now write 
\begin{align*}
 k_j^{(2)}(x,x;\mu)&=k_j^{(2a)}(x,x;\mu)+k_j^{(2b)}(x,x;\mu)\\
 &=\int_{|\xi|\le 1}\chi(\xi)a_j(x,\xi;\mu)\,\dbar\xi + 
 \int_{1\le |\xi|\le\mu}a_j(x,\xi;\mu)\,\dbar\xi.
\end{align*}
By \eqref{eq:kernel01} we obtain immediately that 
 $$k_j^{(2a)}(x,x;\mu)=\sum_{\ell=0}^{L-1} c^{\prime\prime}_{j,\ell}(x)\mu^{d-\nu-\ell}
    +O(\mu^{d-\nu-L}).$$
By homogeneity for $|\xi|\ge1$ of the $q_{j,\ell}$ and by using polar-coordinates, 
 $$\int_{1\le |\xi|\le\mu}q_{j,\ell}(x,\xi)\,\dbar\xi=
     \begin{cases}
      c_{j,\ell}^\prime(x)(\mu^{\nu-j+\ell+n}-1) &: \nu-j+\ell+n\not=0,\\
      c_{j,\ell}^\prime(x)\log\mu &: \nu-j+\ell+n=0
     \end{cases}
 $$
By the second line in \eqref{eq:kernel01} and the homogeneity of $s_{j,L}$, 
\begin{align}\label{eq:kernel03}
 s_{j,L}(x,\xi;\mu)=|\xi|^{d-j}s_{j,L}\Big(x,\frac{\xi}{|\xi|};\frac{\mu}{|\xi|}\Big)
 =O(\mu^{d-\nu-L}|\xi|^{\nu-j+L}),\qquad |\xi|\ge 1.
\end{align}
If $s^h_{j,L}$ denotes the extension by homogeneity of $s_{j,L}$ from $|\xi|\ge 1$ to all 
$\xi\not=0$  $($defined by the second term in \eqref{eq:kernel03}$)$, then 
 $$s^h_{j,L}(x,\xi;\mu)=O(\mu^{d-\nu-L}|\xi|^{\nu-j+L}),\qquad \xi\not=0.$$
Then 
\begin{align*}
 \int_{1\le |\xi|\le\mu}s_{j,L}(x,\xi;\mu)\,\dbar\xi
 &=\int_{0\le |\xi|\le\mu}s^h_{j,L}(x,\xi;\mu)\,\dbar\xi
    -\int_{0\le|\xi|\le1}s^h_{j,L}(x,\xi;\mu)\,\dbar\xi\\
 &= c_{j,L}(x)\mu^{d-j+n}+O(\mu^{d-\nu-L}).
\end{align*}
This yields the expansion of $k_j^{(2b)}(x,x;\mu)$ and completes the proof. 
\end{proof}

\subsection{Application to the resolvent of a $\psi$do}
\label{sec:06.2}

Assume we are given two $\psi$do, $p(x,\xi)\in S^m(\rz^n)$ of 
positive integer order $m\in\nz$ and $q(x,\xi)\in S^\omega(\rz^n)$ with $\omega\in\rz$. 
Moreover, let 
 $$\Lambda=\{\mu e^{i\theta}\mid \mu\ge0,\; 0\le|\theta|\le \Theta\},\qquad 0<\Theta<\pi,$$
be a sector in the complex plane. Then, for every $\theta$,   
 $$a_\theta(x,\xi;\mu):=\mu^m-e^{-i\theta}p(x,\xi)\in \bfS^{m,m}.$$
Note that $e^{i\theta}a_\theta(x,\xi;r^{1/m})=re^{i\theta}-p(x,\xi)$. Now assume that $a_\theta$ 
is elliptic, uniformly with respect to $\theta$, i.e., 
 $$|(\mu^m-e^{-i\theta}p_0(x,\xi))^{-1}|\lesssim 1,\qquad |\xi,\mu|=1,$$
uniformly in $x\in\rz^n$ and $0\le|\theta|\le\Theta$. 
Using Theorem \ref{thm:param02}, there exists a $b_\theta(x,\xi;\mu)\in \bfS^{-m,m}$, 
depending uniformly on $\theta$, such that $a_\theta(x,D;\mu)$ is invertible for large $\mu$ with 
$a_\theta(x,D;\mu)^{-1}=b_\theta(x,D;\mu)$. 
We then find, for every positive integer $\ell$,   
 $$c_{\theta}(x,D;\mu):=q(x,D)\big(\mu^m e^{i\theta}-p(x,D)\big)^{-\ell}
    =e^{-i\theta \ell}q(x,D)b_\theta(x,D;\mu)^\ell.$$
Note that the $\ell$-fold Leibniz product of $b_\theta$ belongs to 
$\bfS^{-m\ell,m}=\wtbfS^{-m\ell,m}+ S^{-m\ell}$. 
Since $S^{-m\ell}\subset \wtbfS^{-m\ell,0}$, we find that 
$c_{\theta}=c_{\theta}^{(1)}+c_{\theta}^{(2)}$ with 
 $$c_{\theta}^{(1)}(x,\xi;\mu)\in \wtbfS^{\omega-m\ell,\omega+m},\qquad 
    c_{\theta}^{(2)}(x,\xi;\mu)\in \wtbfS^{\omega-m\ell,\omega},$$
with uniform dependence on $\theta$. If $\ell$ is so large that $\omega-m\ell<-n$, we can apply 
Theorem \ref{thm:trace-expansion-01} to both $c_{\theta}^{(1)}$ and $c_{\theta}^{(2)}$. 
This is the key to obtain the following: 

\begin{theorem}\label{thm:trace-expansion-main}
With the above notation and assumptions, let $k(x,y;\lambda)$ be the  
distributional kernel of $q(x,D)\big(\lambda-p(x,D)\big)^{-\ell}$. 
Then there exist $\scrC^\infty_b$-functions $c_j(x)$, $c_j^\prime(x)$, $c_j^{\prime\prime}(x)$, 
$j\in\nz_0$, such that 
\begin{equation}\label{eq:exp}
 k(x,x;\lambda)\sim \sum_{j=0}^{+\infty}c_j(x)\lambda^{\frac{n+\omega-j}{m}-\ell}
   +\sum_{j=0}^{+\infty}
    \big(c_j^\prime(x)\log\lambda+c_j^{\prime\prime}(x)\big)\lambda^{-\ell-\frac{j}{m}},
\end{equation}
uniformly for $\lambda\in \Lambda$ with $|\lambda|\lra+\infty$. 
Moreover, $c_j^\prime=c_j^{\prime\prime}\equiv0$ whenever $j$ is not an integer 
multiple of $m$. 
\end{theorem}

\begin{proof}[Proof of Theorem $\ref{thm:trace-expansion-main}$]
Applying Theorem \ref{thm:trace-expansion-01} to both $c_{\theta}^{(1)}$ and $c_{\theta}^{(2)}$ 
one obtains an expansion 
\begin{equation*}
 k(x,x;\mu^m e^{i\theta})
 \sim \sum_{j=0}^{+\infty}\wt{c}_j(x,\theta)\mu^{n+\omega-j-m\ell}
   +\sum_{j=0}^{+\infty}  \big(\wt{c}_j^\prime(x,\theta)\log \mu
   +\wt{c}_j^{\prime\prime}(x,\theta)\big)\mu^{-\ell m-j},
\end{equation*}
for $\mu\to+\infty$, uniformly in $\theta$. Writing  
$\log\mu=\log(\mu e^{i\theta})-i\theta$, $\mu^a=(\mu e^{i\theta})^a (e^{-i\theta})^a$, 
and substituting $\mu=r^{1/m}$ yields the expansion \eqref{eq:exp}, but with coefficient functions 
depending on $\theta$. However, due to the holomorphy of the left-hand side $($for fixed $x)$, 
the coefficients must be constant in $\theta$ as shown in \cite[Lemma 2.3]{GrSe}. 

To see that the coefficients $c_j^\prime$ and $c_j^{\prime\prime}$ vanish whenever $j$ is not 
an integer multiple of $\ell$, one needs to repeat the considerations from \cite[Section 2.2]{GrSe} 
concerning the construction of the parametrix of $\mu^m-p(x,\xi)$.
\end{proof}

\section{Operators on manifolds}
\label{sec:07}

We shall show that the various symbol classes introduced so far lead to corresponding 
operator-classes on smooth compact manifolds. In particular, we shall show that the 
expansion at infinity and the concept of principal limit-symbol extend to the global setting. 

\subsection{Invariance under change of coordinates}
\label{sec:07.1}

Let $\kappa:\rz^n\to\rz^n$ be a smooth change of coordinates and assume that 
$\partial_j \kappa_k\in\scrC^\infty_b(\rz^n)$ for all $1\le j,k\le n$, and that 
$|\mathrm{det}\,\kappa^\prime|$ is uniformly bounded from above and below by
positive constants; here, $\kappa^\prime$ denotes the first derivative 
$($Jacobian matrix$)$ of $\kappa$.
For an operator $A:\scrS(\rz^n)\to \scrS(\rz^n)$ its push-forward $\kappa_*A$ is
defined by
 $$(\kappa_*A) u=[A(u\circ\kappa)]\circ\kappa^{-1},\qquad u\in\scrS(\rz^n).$$
Its pull-back is $\kappa^*A:=(\kappa^{-1})_*A$. If $A(\mu)$ is depending on a
parameter $\mu$, pull-back
and push-forward are defined in the same way, resulting in families
$\kappa^*A(\mu)$ and $\kappa_*A(\mu)$,
respectively. 
It is then well-known that the classes $S^d_{1,0}$ and $S^d$ are invariant under the change 
of coordinates $x=\kappa(y)$. 

\begin{theorem}
The classes $\wt{S}^{d,\nu}_{1,0}$, $\wt{S}^{d,\nu}$, $\wtbfS^{d,\nu}_{1,0}$,
and $\wtbfS^{d,\nu}$
are invariant under the change of coordinates $x=\kappa(y)$. In the classes of
poly-homogeneous symbols,
the homogeneous principal symbols satisfy the $($usual$)$ relation  
$$(\kappa_*a)_0(x,\xi;\mu)=a_0\big(\kappa^{-1}(x),\kappa^\prime(\kappa^{-1}(x))^{t}\xi;\mu\big),$$
where $\kappa^\prime(y)^{t}$ denotes the adjoint of the first derivative
$\kappa^\prime(y)$.
\end{theorem}
\begin{proof}
In Theorem 2.1.21 of \cite{Grub} the invariance is shown for the classes $S^{d,\nu}_{1,0}$ 
and $S^{d,\nu}$. This includes the classes $\wt{S}^{d,\nu}_{1,0}$ and $\wt{S}^{d,\nu}$ for 
$\nu\le 0$. If $\nu>0$, we choose a symbol $p(x,\xi)\in S^{-\nu}(\rz^n)$ which has 
inverse $q(x,\xi)\in S^{\nu}(\rz^n)$ with respect to the Leibniz product.  
Given $a(x,\xi;\mu)\in\wt{S}^{d,\nu}_{1,0}$, we find  
 $$\kappa_*a=\kappa_*(a\#p)\#\kappa_*q \in\wt{S}^{d,\nu}_{1,0},$$
since $a\#p\in \wt{S}^{d-\nu,0}_{1,0}$. Analogously we argue for $\wt{S}^{d,\nu}$. 

Next let $a(x,\xi;\mu)\in\wtbfS^{d,\nu}_{1,0}$ be as in Definition \ref{def:symbol-expansion}. 
The invariance follows from the observation that the classes $S^{\nu+j}_{1,0}(\rz^n)$ 
and $\wt{S}^{d,\nu+N}_{1,0}$ are invariant, while 
$\kappa_*[\xi,\mu]^{d-\nu-j}\in S^{d-\nu-j}\subset \wtbfS^{d-\nu-j,0}_{1,0}$ 
has a complete expansion due to Proposition \ref{prop:classical}. This allows to find the 
complete expansion of $\kappa_*a(x,\xi;\mu)$. 
Using the formula for the asymp\-totic expansion of $\kappa_*a$, one sees that 
poly-homogeneous symbols remain poly-homogeneous. 	
\end{proof}

Let us have a closer look to the homogeneous principal symbol of $a\in\wtbfS^{d,\nu}$. 
For convenience of notation let us set
$p(x,\xi;\mu)=(\kappa_*a)_0(x,\xi;\mu)$
and $\calK(x)=\kappa^\prime(\kappa^{-1}(x))^{t}$. To see that $p$ belongs to
$\wtbfS^{d,\nu}_{hom}$ we write
 $$p(x,\xi;\mu)=|\xi,\mu|^d\, \wh{p}\Big(x,\frac{(\xi,\mu)}{|\xi,\mu|}\Big),$$
where, in polar-coordinates, 
$$\wh{p}(x,r,\phi)=p(x,r\phi,\sqrt{1-r^2})={a}\big(\kappa^{-1}(x),r\calK(x)\phi,\sqrt{1-r^2}\big).$$
Introducing 
\begin{align*}
n(x,r,\phi)^2&=\big|r\calK(x)\phi,\sqrt{1-r^2}\big|^2=1-r^2\big(1-|\calK(x)\phi|^2\big),\\
 s(x,r,\phi)&=r|\calK(x)\phi|/n(x,r,\phi),\\
 \theta(x,\phi)&=\calK(x)\phi/|\calK(x)\phi|,
\end{align*}
we find  
 $$\wh{p}(x,r,\phi)=n^d\, \wh{a}\big(\kappa^{-1}(x),s\theta,\sqrt{1-s^2}\big).$$
Noting that $n$ is smooth in $r$ up to $r=0$ and using the weigthed
Taylor-expansion of $\wh{a}$, one finds that
$\wh{p}$ admits a weighted Taylor-expansion with principal angular symbol 
\begin{align*}
 \wh{p}_{\spk{\nu}}(x,\phi)
 &=\lim_{r\to 0+}n^d(r/s)^{-\nu}s^{-\nu}\, \wh{a}\big(\kappa^{-1}(x),s\theta,\sqrt{1-s^2}\big)\\
 &=|\calK(x)\phi|^\nu\, a_{\spk{\nu}}(\kappa^{-1}(x),\theta).
\end{align*}
This results in the following observation: 

\begin{proposition}\label{prop:angular symbol-change}
Let $a\in \wtbfS^{d,\nu}$. The principal angular symbols of $a$ and $\kappa_*a$
satisfy the relation
$$(\kappa_*
a)_{\spk{\nu}}(x,\xi)=a_{\spk{\nu}}\big(\kappa^{-1}(x),\kappa^\prime(\kappa^{-1}(x))^{t}\xi\big).$$
\end{proposition} 

In other words, the principal angular symbol transforms as a function on the
cotangent-bundle of $\rz^n$.

\begin{remark}
In the above discussion we have focused on changes of coordinates defined
on $\rz^n$, satisfying certain growth conditions at infinity. 
This is the natural setting for symbols which are globally defined on $\rz^n$. 
Alternatively, we could consider arbitrary diffeomorphisms $\kappa:U\to V$ with 
arbitrary open subsets $U$, $V$ of $\rz^n$ and the push-forward of
$\psi$do of the form
$\phi \, a(x,D;\mu)\psi$ with $\phi,\psi\in\scrC^\infty_{\mathrm{comp}}(U)$. 
We would obtain a corresponding invariance property; the details are left to the reader.
\end{remark}

The invariance under changes of coordinates permits to define corresponding
classes for manifolds.

\begin{definition}\label{def:pseudo-manifold}
Let $M$ be a smooth closed manifold. With
$\wt{L}^{d,\nu}_{1,0}=\wt{L}^{d,\nu}_{1,0}(M;\rpbar)$ we denote the space
of all operator-families $A(\mu):\scrC^\infty(M)\to\scrC^\infty(M)$ with the
following property: Given an arbitrary chart $\kappa:\Omega\subset M\to
U\subset\rz^n$ and arbitrary functions $\phi,\psi\in\scrC^\infty_{\mathrm{comp}}(\Omega)$, 
the operator-family
$\kappa_*(\phi A(\mu)\psi)$ defined by 
$$u\mapsto \kappa_*(\phi
A(\mu)\psi)u=[\phi\,A(\mu)(\psi(u\circ\kappa))]\circ\kappa^{-1} ,
    \qquad u\in\scrS(\rz^n),\footnotemark$$
is\footnotetext{In this definition, smooth functions with compact support in
some open set are considered as functions on the whole ambient space after
extension by zero.} a $\psi$do with symbol from
$\wt{S}^{d,\nu}$.
Analogously, we define the spaces $\wt{L}^{d,\nu}=\wt{L}^{d,\nu}(M;\rpbar)$,
$\wtbfL^{d,\nu}_{1,0}=\wtbfL^{d,\nu}_{1,0}(M;\rpbar)$, and
$\wtbfL^{d,\nu}=\wtbfL^{d,\nu}(M;\rpbar)$.
\end{definition} 

In $\wtbfL^{d,\nu}$ both homogeneous principal symbol and principal angular symbol are well 
defined functions on $(T^*M\setminus0)\times\rpbar$ and $T^*M\setminus0$, 
respectively.
Let us mention that 
$\wt{L}^{d-\infty,\nu-\infty}_{1,0}
=\wt{L}^{d-\infty,\nu-\infty}=S^{d-\nu}_{1,0}(\rpbar,L^{-\infty}(M))$.

Proceeding as usual, one can show that
any of the four classes is closed under composition and, after fixing an
arbitrary Riemannian metric on $M$ which allows
the definition of a corresponding space $L^2(M)$ of square integrable functions,
under taking the formal adjoint:

\begin{theorem}
Composition of operator-families induces a map
$\wt{L}^{d_1,\nu_1}_{1,0}\times \wt{L}^{d_0,\nu_0}_{1,0}\to
\wt{L}^{d_0+d_1,\nu_0+\nu_1}_{1,0}$; taking the formal adjoint
induces a map
$\wt{L}^{d,\nu}_{1,0}\to \wt{L}^{d,\nu}_{1,0}$. Analogous results hold for the
three other classes introduced in Definition
$\ref{def:pseudo-manifold}$.  
\end{theorem}

For an alternative description of the operator-classes, let us choose a system of charts 
$\kappa_i:\Omega_i\to U_i$, $i=1,\ldots,m$, such that the $\Omega_i$ cover $M$; 
moreover let $\phi_i,\psi_i\in\scrC^\infty(\Omega_i)$ such that the $\phi_i$ are a partition 
of unity and $\psi_i\equiv1$ in a neighborhood of the support of $\phi_i$. 
Then $\wt{L}^{d,\nu}_{1,0}$ consists of all operators of the form 
 $$A(\mu)=\sum_{i=1}^m 
     \kappa_i^*\big((\phi_i\circ\kappa_i^{-1})\,a_i(x,D;\mu)\,(\psi_i\circ\kappa_i^{-1})\big)
     \mod \wt{L}^{d-\infty,\nu-\infty}_{1,0}$$
with $a_i\in \wt{S}^{d,\nu}_{1,0}$. The analogous statement holds for the other classes. 

\subsection{Complete expansion and limit operator}
\label{sec:07.2}

The extension of the concept of complete expansion and principal limit-symbol to manifolds 
requires some additional analysis. The key is to show that the symbol $[\xi,\mu]^\alpha$ involved 
in the definition of $\wtbfS^{d,\nu}$ can be replaced by other ones. 

It is convenient to use the notation $\bflambda^\alpha(\xi,\mu)=[\xi,\mu]^\alpha$, $\alpha\in\rz$. 
Then the expansion of a symbol $a(x,\xi;\mu)\in \wtbfS^{d,\nu}_{1,0}$ takes the form  
$$a=\sum_{j=0}^{N-1}a_{[\nu+j]}^\infty\#\bflambda^{d-\nu-j} \mod
\wt{S}^{d,\nu+N}_{1,0};$$
note that here the Leibniz product actually coincides with the point-wise product of the involved 
symbols. 

\begin{definition}\label{def:order-reducing-family}
A family of order-reducing symbols is a set
$\Lambda=\{\lambda^\alpha(x,\xi;\mu)\mid \alpha\in\rz\}$
of symbols $\lambda^\alpha\in S^\alpha$ which satisfy 
\begin{itemize}
 \item[$(1)$] $\lambda^0=1\mod S^{-1}$, 
\item[$(2)$] $\lambda^\alpha\#\lambda^{\beta}=\lambda^{\alpha+\beta}\mod
S^{\alpha+\beta-1}$
  for every $\alpha,\beta\in\rz$. 
\end{itemize} 
\end{definition}

Note that any $\lambda^\alpha$ in such a family is parameter-elliptic in
$S^{\alpha}$ and thus has a parametrix in $S^{-\alpha}$; this parametrix
coincides with $\lambda^{-\alpha}$ modulo $S^{-\alpha-1}$.

\begin{theorem}\label{thm:expansion}
Let $\Lambda$ be a family of order-reducing symbols as in Definition
$\ref{def:order-reducing-family}$. Then
for a symbol $a(x,\xi;\mu)\in\wt{S}^{d,\nu}_{1,0}$ the following are equivalent:\begin{itemize}
\item[$a)$] $a\in \wtbfS^{d,\nu}_{1,0}$ $($cf. Definition
$\ref{def:symbol-expansion})$.
 \item[$b)$] There exist  
$a^{\Lambda,\infty}_{[\nu+j]}(x,\xi)\in S^{\nu+j}_{1,0}(\rz^n)$
such that, for every $N\in\nz$, 
     $$a=\sum_{j=0}^{N-1}a^{\Lambda,\infty}_{[\nu+j]}\#\lambda^{d-\nu-j} 
         \mod  \wt{S}^{d,\nu+N}_{1,0}$$ 
 \end{itemize}
If $a^{\infty}_{[\nu]}(x,\xi)$ is the principal limit symbol of $a$ then 
$$a^{\Lambda,\infty}_{[\nu]}=a^{\infty}_{[\nu]}\#\lambda^{-(d-\nu)}_0(x,0,1),$$
where $\lambda^\alpha_0(x,\xi;\mu)$ denotes the homogeneous principal symbol of
$\lambda^\alpha$.
\end{theorem}

Before coming to the proof, let us show that the coefficients in any expansion of 
Theorem \ref{thm:expansion}.b$)$ are uniquely determined: 
Suppose $a=0$ and that we already have verified that $a^{\Lambda,\infty}_{[\nu+j]}=0$ 
for $j=0,\ldots,N-1$. Then 
$a^{\Lambda,\infty}_{[\nu+N]}\#\lambda^{d-\nu-N}\in \wt{S}^{d,\nu+N+1}_{1,0}$. 
Composing from the right with $\lambda^{-(d-\nu-N)}$ one finds 
 $$a^{\Lambda,\infty}_{[\nu+N]}(x,\xi)=(a^{\Lambda,\infty}_{[\nu+N]}\#r_0)(x,\xi;\mu) + 
    r_1(x,\xi;\mu)$$
with some $r_0\in S^{-1}$ and $r_1\in\wt{S}^{\nu+N,\nu+N+1}_{1,0}$. The right-hand 
side decays as $1/\mu$ in any semi-norm of ${S}^{\nu+N+1}_{1,0}(\rz^n)$. Thus 
$a^{\Lambda,\infty}_{[\nu+N]}=0$. 

\begin{proof}[Proof of Theorem $\ref{thm:expansion}$]
First we argue that we may assume without loss of generality that $\nu=0$. 
To this end let $p_s(\xi):=\spk{\xi}^s$, $s\in\rz$. Then $p_{-\nu}\#a\in
\wt{S}^{d^\prime,0}_{1,0}$
for $d^\prime=d-\nu$.  

Given hypothesis a$)$, then $p_{-\nu}\#a\in \wtbfS^{d^\prime,0}_{1,0}$ and we
show the existence
of an expansion 
 $$p_{-\nu}\#a=\sum_{j=0}^{N-1}b^{\Lambda,\infty}_{[j]}\#\lambda^{d^\prime-j} 
         \mod  \wt{S}^{d^\prime,N}_{1,0}.$$
Multiplying from the left with $p_\nu$ we find the desired expansion for $a$
with
$a^{\Lambda,\infty}_{[\nu+j]}:=p_\nu\#b^{\Lambda,\infty}_{[j]}$. 
We argue similarly when starting out from hypothesis b$)$. 

Now let $\nu=0$; we show that b$)$ implies a$)$.  
By Proposition \ref{prop:classical}, $\lambda^{d-j}\in\wtbfS^{d-j,0}_{1,0}$ has an
expansion
$$\lambda^{d-j}=\sum_{\ell=0}^{N-1}b^{\infty}_{j,[\ell]}\#\bflambda^{d-j-\ell}         
   \mod  \wt{S}^{d-j,N}_{1,0};$$
in particular, $b^{\infty}_{j,[0]}(x)=\lambda^{d-j}(x,0;1)$. Therefore  
 $$\sum_{j=0}^{N-1}a^{\Lambda,\infty}_{[j]}\#\lambda^{d-j-\ell} 
    =\sum_{j=0}^{N-1}\sum_{\ell=0}^{N-1}
       a^{\Lambda,\infty}_{[j]}\#b^{\infty}_{j,[\ell]}\#\bflambda^{d-j-\ell}
       \mod  \wt{S}^{d,N}_{1,0},$$
since 
$a^{\Lambda,\infty}_{[j]}\#\wt{S}^{d-j,N}_{1,0}
\subset \wt{S}^{d,N+j}_{1,0}\subset \wt{S}^{d,N}_{1,0}$ for every $j$. If
$m:=j+\ell\ge N$,
 $$a^{\Lambda,\infty}_{[j]}\#b^{\infty}_{j,[\ell]}\#\bflambda^{d-j-\ell}\in 
    S^m_{1,0}(\rz^n)\# S^{d-m}\subset 
    \wt{S}^{m,m}_{1,0}\#\wt{S}^{d-m,0}_{1,0}\subset \wt{S}^{d,m}_{1,0}
    \subset \wt{S}^{d,N}_{1,0}.$$
We thus find 
     $$a=\sum_{k=0}^{N-1}a^{\infty}_{[k]}\#\bflambda^{d-k} 
         \mod  \wt{S}^{d,N}_{1,0},\qquad 
a^{\infty}_{[k]}:=\sum_{j+\ell=k}
a^{\Lambda,\infty}_{[j]}\#b^{\infty}_{j,[\ell]}
         \in S^k_{1,0}(\rz^n).$$
In particular, $a^{\infty}_{[0]}=a^{\Lambda,\infty}_{[0]}\#b^{\infty}_{0,[0]}
=a^{\Lambda,\infty}_{[0]}\#\lambda^d_0(x,0;1)$. 

Next we show that a$)$ implies b$)$ $($again with $\nu=0)$. We start out from
the expansion
     $$a=\sum_{j=0}^{N-1}a^{\infty,0}_{[j]}\#\bflambda^{d-j} 
         \mod  \wt{S}^{d,N}_{1,0};$$ 
the additional super-script $0$ is introduced for systematic reasons, since we
will now establish an iterative procedure to transform this expansion in an
expansion using the family $\Lambda$. Write
\begin{align}\label{eq:step1}
a^{\infty,0}_{[0]}\#\bflambda^{d}=a^{\infty,0}_{[0]}\#(\bflambda^{d}\#\lambda^{-d})\#\lambda^d
 +a^{\infty,0}_{[0]}\# r_0
\end{align}
with $r_0\in S^{d-1}$. By Proposition \ref{prop:classical} we have expansions 
\begin{align*}
 \bflambda^{d}\#\lambda^{-d}
&=\sum_{j=0}^{N-1}b^{\infty}_{[j]}\#\bflambda^{-j}  \mod  \wt{S}^{0,N}_{1,0},\\ r_0
 &=\sum_{j=0}^{N-2}r^{\infty}_{[j]}\#\bflambda^{d-1-j}  
 \mod  \wt{S}^{d-1,N-1}_{1,0}\subset \wt{S}^{d,N}_{1,0}. 
\end{align*}
Inserting this in the expansion \eqref{eq:step1} and using for $j\ge 1$
expansions
\begin{align*}
 \bflambda^{-j}\#\lambda^{d}
 &=\sum_{\ell=0}^{N-2}c^{\infty}_{j,[\ell]}\#\bflambda^{d-j-\ell}  
\mod
\wt{S}^{d-j,N-1}_{1,0}\subset\wt{S}^{d,N-1+j}_{1,0}\subset\wt{S}^{d,N}_{1,0},
\end{align*}
we find 
\begin{align*}
 \sum_{j=0}^{N-1}a^{\infty,0}_{[j]}\#\bflambda^{d-j}
  =&\; a^{\infty,0}_{[0]}\#b^\infty_{[0]}\#\lambda^{d}
   +\sum_{j=1}^{N-1}\sum_{\ell=0}^{N-2}
a^{\infty,0}_{[0]}\#b^\infty_{[j]}\#c^\infty_{j,[\ell]}\#\bflambda^{d-j-\ell}\\&
+\sum_{j=0}^{N-2}(a^{\infty,0}_{[j+1]}+a^{\infty,0}_{[0]}\#r^\infty_{[j]})\#\bflambda^{d-1-j}
\mod \wt{S}^{d,N}_{1,0}.
\end{align*}
The second term on the right-hand side equals 
\begin{align*}
   \sum_{j=0}^{N-2}\sum_{\ell=0}^{N-2}&
a^{\infty,0}_{[0]}\#b^\infty_{[j+1]}\#c^\infty_{j+1,[\ell]}\#\bflambda^{d-1-j-\ell}
\\
   &=\sum_{k=0}^{N-2}\Big(\sum_{j+\ell=k}^{N-2}
a^{\infty,0}_{[0]}\#b^\infty_{[j+1]}\#c^\infty_{j+1,[\ell]}\Big)\#\bflambda^{d-1-k}
   \mod \wt{S}^{d,N}_{1,0}.
\end{align*}
We conclude that 
\begin{align*}
 a=\sum_{j=0}^{N-1}a^{\infty,0}_{[j]}\#\bflambda^{d-j}
  = a^{\Lambda,\infty}_{[0]}\#\lambda^{d}
+\sum_{j=0}^{N-2} a^{\infty,1}_{[j+1]}\#\bflambda^{d-1-j} \mod
\wt{S}^{d,N}_{1,0}.
\end{align*}
with $a^{\Lambda,\infty}_{[0]}=a^{\infty,0}_{[0]}\#b^\infty_{[0]}$ and resulting
symbols
$a^{\infty,1}_{[j+1]}\in S^{j+1}_{1,0}(\rz^n)$. This finishes the
first step of the procedure.
In the second step we write 
 $$a^{\infty,1}_{[1]}\#\bflambda^{d-1}
     =a^{\infty,1}_{[1]}\#(\bflambda^{d-1}\#\lambda^{-(d-1)})\#\lambda^{d-1}+
       a^{\infty,1}_{[1]}\#r_1$$
with $r_1\in S^{d-2}$ and proceed as above to finally obtain 
\begin{align*}
 \sum_{j=0}^{N-2}a^{\infty,1}_{[j+1]}\#\bflambda^{d-1-j}
  = a^{\Lambda,\infty}_{[1]}\#\lambda^{d-1}
+\sum_{j=0}^{N-3} a^{\infty,2}_{[j+2]}\#\bflambda^{d-2-j} \mod
\wt{S}^{d,N}_{1,0}.
\end{align*}
with resulting $a^{\Lambda,\infty}_{[1]}$ and 
$a^{\infty,2}_{[j+2]}\in S^{j+2}_{1,0}(\rz^n)$, hence 
\begin{align*}
a= a^{\Lambda,\infty}_{[0]}\#\lambda^{d}
+a^{\Lambda,\infty}_{[1]}\#\lambda^{d-1}
+\sum_{j=0}^{N-3} a^{\infty,2}_{[j+2]}\#\bflambda^{d-2-j} \mod
\wt{S}^{d,N}_{1,0}.
\end{align*}
We iterate this procedure until the $N$-th step which consists in writing 
\begin{align*}
 \sum_{j=0}^{0}a^{\infty,N-1}_{[j+N-1]}\#\bflambda^{d-(N-1)-j}
  = a^{\infty,N-1}_{[N-1]}\#\bflambda^{d-(N-1)}
  = a^{\Lambda,\infty}_{[N-1]}\#\lambda^{d-(N-1)} \mod \wt{S}^{d,N}_{1,0}.
\end{align*}
resulting in 
\begin{align*}
a=\sum_{j=0}^{N-1} a^{\Lambda,\infty}_{[j]}\#\lambda^{d-j}\mod
\wt{S}^{d,N}_{1,0}
\end{align*}
as claimed in b$)$. The proof is complete. 
\end{proof}

The following lemma will be useful in discussing localizations of operator-families. 

\begin{lemma}\label{lem:locality}
Let $a(x,\xi;\mu)\in \wtbfS^{d,\nu}_{1,0}$ have the expansion 
\begin{align*}
 a=\sum_{j=0}^{N-1} a^{\Lambda,\infty}_{[\nu+j]}\#\lambda^{d-\nu-j}\mod\wt{S}^{d,N}_{1,0}
\end{align*}
with respect to some family of order-reducing symbols $\Lambda$. Let $K\subset\rz^n$ be 
a compact set and assume that $a\#\phi=a$ for every function 
$\phi\in\scrC^\infty_{\mathrm{comp}}(\rz^n)$ 
with $\phi\equiv 1$ in an open neighborhood of $K$. Then, for every such function $\phi$ and 
every $j\ge0$, 
\begin{align}\label{eq:coeff01}
 a^{\Lambda,\infty}_{[\nu+j]}\#\phi=a^{\Lambda,\infty}_{[\nu+j]}.
\end{align}
The analogous result for left-multiplication with $\phi$ holds also true $($and follows trivially from the uniqueness of the coefficient-symbols in the expansion$)$.  
\end{lemma}
\begin{proof}
We proceed by induction. Since 
 $$a-a^{\Lambda,\infty}_{[\nu]}\#\lambda^{d-\nu}\in \wt{S}^{d,\nu+1}_{1,0},$$
multiplication from the right with $\phi$ yields 
 $$a-a^{\Lambda,\infty}_{[\nu]}\#\phi\#\lambda^{d-\nu}
    +a^{\Lambda,\infty}_{[\nu]}\#[\lambda^{d-\nu},\phi]
    \in \wt{S}^{d,\nu+1}_{1,0},$$
where $[\cdot,\cdot]$ is the commutator $($with respect to $\#)$. Now the third term belongs to 
 $$S^\nu(\rz^n)\#S^{d-\nu-1}\subset
     \wt{S}^{\nu,\nu}_{1,0}\#\wt{S}^{d-\nu-1}_{1,0}\subset \wt{S}^{d-1,\nu}_{1,0}
     \subset \wt{S}^{d,\nu+1}_{1,0}.$$
Therefore 
 $$a-(a^{\Lambda,\infty}_{[\nu]}\#\phi)\#\lambda^{d-\nu} \in \wt{S}^{d,\nu+1}_{1,0}.$$
The uniqueness of the coefficients in the expansion then implies    
$a^{\Lambda,\infty}_{[\nu]}\#\phi=a^{\Lambda,\infty}_{[\nu]}$. 

Now suppose that \eqref{eq:coeff01} holds for $j=0,\ldots,N-1$. Given a function $\phi$ choose 
$\phi_0\in\scrC^{\infty}_0(\rz^n)$ such that $\phi_0\equiv1$ near $K$ and $\phi\equiv 1$ 
near the support of $\phi_0$. Then, by induction assumption, we have 
\begin{align*}
 a=a\#\phi
 =&\, \sum_{j=0}^{N-1} a^{\Lambda,\infty}_{[\nu+j]}\#\phi_0\#\lambda^{d-\nu-j}\#\phi\\
   & +a^{\Lambda,\infty}_{[\nu+N]}\#\phi\#\lambda^{d-\nu-N}
     +a^{\Lambda,\infty}_{[\nu+N]}\#[\lambda^{d-\nu-N},\phi]
     \mod \wt{S}^{d,\nu+N+1}_{1,0}.
\end{align*}
As above, the last term is shown to be in $\wt{S}^{d,\nu+N+1}_{1,0}$. Moreover, 
$\phi_0\#\lambda^{d-\nu-j}\#(1-\phi)$ belongs to $S^{-\infty}$. Using again the induction hypotheses 
we  derive  
\begin{align*}
 a=\sum_{j=0}^{N-1} a^{\Lambda,\infty}_{[\nu+j]}\#\lambda^{d-\nu-j}
 +(a^{\Lambda,\infty}_{[\nu+N]}\#\phi)\#\lambda^{d-\nu-N}
     \mod \wt{S}^{d,\nu+N+1}_{1,0}.
\end{align*}
Thus, by uniqueness of the coefficients, \eqref{eq:coeff01} holds $j=N$.
\end{proof}

Now lets turn to the global situation of operators on the manifold $M$. 
Let us fix some Riemannian metric $g$ on $M$. 

\begin{definition}
A family of order-reducing operators on $M$ is a set 
$\Lambda=\{\Lambda^\alpha(\mu)\mid\alpha\in\rz\}$ where 
$\Lambda^\alpha(\mu)\in L^\alpha$ has homogeneous principal symbol 
 $$\sigma^\alpha(\Lambda^\alpha)(v;\mu)=\big(|v|^2+\mu^2\big)^{\alpha/2}$$
and $\Lambda^0=1$ $(|v|$ denotes the modulus of a co-vector $v\in T^*M$ with 
respect to $g)$. 
\end{definition}

\begin{theorem}\label{thm:expansion-manifold}
Let $A(\mu)\in\wtbfL^{d,\nu}_{1,0}$. Then there exists uniquely determined operators 
$A^\infty_{[\nu+j]}\in L^\nu_{1,0}(M)$, $j\in\nz_0$, such that, for every $N\in\nz$, 
 $$A(\mu)=\sum_{j=1}^{N-1}A^\infty_{[\nu+j]}\Lambda^{d-\nu-j}(\mu)
    \mod \wt{L}^{d,\nu+N}_{1,0}.$$
The leading coefficient $A^\infty_{[\nu]}$ is called the limit-operator of $A(\mu)$. 
\end{theorem}
\begin{proof}
The proof of the uniqueness is analogous to the one given after Theorem \ref{thm:expansion}. 
Therefore we shall focus on the existence of the expansion. 

Let $\Omega_1,\ldots,\Omega_m$ be a covering of $M$ such that any union $\Omega_i\cup\Omega_j$ 
is contained in a chart(-domain) of $M$. Let $\phi_i\in\scrC^\infty_{\mathrm{comp}}(\Omega_i)$, 
$i=1,\ldots,M$, be a sub-ordinate partition of unity. Then $A(\mu)=\sum_{i,j}\phi_i\,A(\mu)\phi_j$. 
It suffices to show the existence of an expansion for each summand. 

Thus we may assume from the beginning that there exist a chart $\kappa:\Omega\to U$ and 
two functions $\phi,\psi\in\scrC^\infty_{\mathrm{comp}}(\Omega)$ such that 
$A(\mu)=\phi A(\mu)\psi$. 
Let $a(x,\xi;\mu)\in\wtbfS^{d,\nu}_{1,0}$ be the symbol of $\kappa_*A(\mu)$ and let 
$K$ be the union of the supports of $\phi\circ\kappa^{-1}$ and $\psi\circ\kappa^{-1}$, 
respectively. $K$ is a compact subset of $U$. 

Let $V$ be an open neighborhood of $K$ with compact closure contained in $U$. 
Take $\theta\in\scrC^\infty_{\mathrm{comp}}(U)$ with $\theta\equiv1$ on $V$ and let 
${\lambda}^\alpha(x,\xi;\mu)\in S^\alpha$ be the symbol of 
$\kappa_*\big((\theta\circ\kappa)\Lambda^\alpha(\mu)(\theta\circ\kappa)\big)$. Note that 
 $${\lambda}^\alpha(x,\xi;\mu)=\theta^2(x)\big(|\xi|_x^2+\mu^2\big)^{\alpha/2}\chi(\xi,\mu)+
     r^\alpha(x,\xi;\mu),$$
where $\chi$ is a zero-excision function and $r^\alpha\in S^{\alpha-1}$. 
Now define 
 $$\wt{\lambda}^\alpha(x,\xi;\mu)=\big(\theta(x)|\xi|_x^2+(1-\theta)(x)|\xi|^2+\mu^2\big)^{\alpha/2}
     \chi(\xi,\mu)+r^\alpha(x,\xi;\mu);$$
then $\wt\Lambda=\{\wt{\lambda}^\alpha\mid\alpha\in\rz\}$ is a family of order-reducing symbols 
in the sense of Definition \ref{def:order-reducing-family} and 
$\wt{\lambda}^\alpha(x,\xi,\mu)={\lambda}^\alpha(x,\xi,\mu)$ whenever $x\in V$. 

By Theorem \ref{thm:expansion} we have an expansion 
$a\sim\sum_j a^{\infty}_{[\nu+j]}\#\wt{\lambda}^{d-\nu-j}$. 
If $\theta_0\in\scrC^\infty_{\mathrm{comp}}(V)$ with $\theta_0\equiv1$ near $K$ then, 
by Lemma \ref{lem:locality},  
$\theta_0a^{\infty}_{[\nu+j]}=a^{\infty}_{[\nu+j]}\#\theta_0=a^{\infty}_{[\nu+j]}$. 
Thus, taking another $\theta_1\in\scrC^\infty_{\mathrm{comp}}(V)$ with $\theta_1\equiv1$ 
near the support of $\theta_0$, we find 
 $$a=\sum_{j=0}^{N-1} a^{\infty}_{[\nu+j]}\#\theta_0\wt{\lambda}^{d-\nu-j}\#\theta_1
     \mod \theta_0\,\wt{S}^{d,\nu+N}_{1,0}\#\theta_1.$$
Since $\theta_0\wt{\lambda}^{d-\nu-j}=\theta_0{\lambda}^{d-\nu-j}$ by construction, 
and applying the pull-back under $\kappa$, we find 
 $$A(\mu)=\sum_{j=1}^{N-1}A^\infty_{[\nu+j]}\,(\theta_0\circ\kappa)\,\Lambda^{d-\nu-j}(\mu)
    \,(\theta_1\circ\kappa)\mod \wt{L}^{d,\nu+N}_{1,0}$$
with $A^\infty_{[\nu+j]}=\kappa^*a^{\infty}_{[\nu+j]}(x,D)$. 
Finally, note that 
$(\theta_0\circ\kappa)\,\Lambda^{d-\nu-j}(\mu)\,(1-\theta_1)\circ\kappa\in L^{-\infty}$ due to 
the disjoint supports of $\theta_0$ and $1-\theta_1$ and that 
$A^\infty_{[\nu+j]}\,(\theta_0\circ\kappa)=A^\infty_{[\nu+j]}$. 
\end{proof}

\begin{example}
If $A(\mu)\in L^d$ then $A(\mu)\in\wtbfL^{d,0}$ as well; 
its limit-operator is the operator of multiplication 
with the function $\sigma(A)(x,0;1)$ $($the homogeneous principal symbol of $A(\mu)$ 
evaluated in $(\xi,\mu)=(0,1))$. 
\end{example} 

\begin{theorem}\label{multiplicativity}
The limit-operator behaves multiplicative under composition: If $A_j(\mu)\in\wtbfL^{d_j,\nu_j}_{1,0}$ 
have limit-operator $A^\infty_{j,[\nu_j]}$ then  
$A_0(\mu)A_1(\mu)\in\wtbfL^{d_0+d_1,\nu_0+\nu_1}_{1,0}$ has the limit-operator 
$A^\infty_{0,[\nu_0]}A^\infty_{1,[\nu_1]}$. 
\end{theorem}
\begin{proof}
In a first step, let $A(\mu)\in\wtbfL^{d,\nu}_{1,0}$ have limit-operator $A^\infty_{[\nu]}$. Then 
  $$A^\infty_{[\nu]}=\lim_{\mu\to+\infty}A(\mu)\Lambda^{\nu-d}(\mu)
     \qquad\text{(convergence in $L^{\nu+1}_{1,0}(M)$)}.$$ 
In fact, using the expansion with $N=1$, 
\begin{align}\label{eq:mult1}
 A(\mu)\Lambda^{\nu-d}(\mu)=A^\infty_{[\nu]}+A^\infty_{[\nu]}R(\mu)
    \mod \wt{L}^{\nu,\nu+1}_{1,0}
\end{align}
with an $R(\mu)\in L^{-1}\subset S^{-1}_{1,0}(\rpbar,L^0_{1,0}(M))$. Then  
$\wt{L}^{\nu,\nu+1}_{1,0} \subset S^{-1}_{1,0}(\rpbar,L^{\nu+1}_{1,0}(M))$ yields the claim. 
Also one sees that $A(\mu)\Lambda^{\nu-d}(\mu)$ is bounded as a function of $\mu$ with values in 
$L^\nu(M)$. 

Since $A_0(\mu)A_1(\mu)\in\wtbfL^{d_0+d_1,\nu_0+\nu_1}_{1,0}$, it suffices to show that 
$A_0(\mu)A_1(\mu)\Lambda^{\nu_0+\nu_1-d_0-d_1}(\mu)$ converges to 
$A^\infty_{0,[\nu_0]}A^\infty_{1,[\nu_1]}$ in $L^m_{1,0}(M)$ for some $m\ge\nu_0+\nu_1+1$. 
Reasoning as before, we see that 
 \begin{align*}
  A_0(\mu)A_1(\mu)\Lambda^{\nu_0+\nu_1-d_0-d_1}(\mu)
  &\equiv A_0(\mu)A_1(\mu)\Lambda^{\nu_1-d_1}(\mu)\Lambda^{\nu_0-d_0}(\mu) \\
  &\equiv A_0(\mu)\Lambda^{\nu_0-d_0}(\mu)
     \Lambda^{d_0-\nu_0}(\mu)A_1(\mu)\Lambda^{\nu_1-d_1}(\mu)\Lambda^{\nu_0-d_0}(\mu) 
 \end{align*}
modulo terms belonging to $S^{-1}_{1,0}(\rpbar,L^{\nu_0+\nu_1}_{1,0}(M))$. 
It remains to show that 
 $$\Lambda^{d_0-\nu_0}(\mu)[A_1(\mu)\Lambda^{\nu_1-d_1}]\Lambda^{\nu_0-d_0}(\mu)
    \xrightarrow{\mu\to+\infty} A^\infty_{1,[\nu_1]} $$
in $L^{m}_{1,0}(M)$ for some $m\ge\nu_1+1$. 
Using the analogue of \eqref{eq:mult1} for $A_1(\mu)$ this is readily seen to be equivalent to 
\begin{equation}\label{eq:conv}
 \Lambda^{d_0-\nu_0}(\mu)A^\infty_{1,[\nu_1]}\Lambda^{\nu_0-d_0}(\mu)
 \xrightarrow{\mu\to+\infty} A^\infty_{1,[\nu_1]}.
\end{equation}
However, from Lemma \ref{lem:mu-expansion} it follows that 
$\mu^{-\alpha}\Lambda^{\alpha}(\mu)$ is bounded in $L^{\alpha_+}_{1,0}(M)$ and, 
for $\mu\to+\infty$,  converges to $1$ in $L^{\alpha_++1}_{1,0}(M)$ for every $\alpha$. 
Therefore, \eqref{eq:conv} holds true with convergence in $L^{\nu_1+|d_0-\nu_0|+1}_{1,0}(M)$. 
\end{proof}

\subsection{Extension to vector-bundles}
\label{sec:07.3}

Given smooth vector-bundles $E_j$, $j=0,1$, on $M$ of dimension $n_0$ and $n_1$, respectively,  
the above definitions and results extend in a straight-forward way to operator-families acting as maps 
$\Gamma(M,E_0)\to\Gamma(M,E_1)$ between the spaces of smooth sections of $E_0$ and $E_1$, 
respectively. The definition of the spaces $\wt{L}^{d,\nu}_{1,0}(E_0,E_1)$, $\wt{L}^{d,\nu}(E_0,E_1)$, 
$\wtbfL^{d,\nu}_{1,0}(E_0,E_1)$, and $\wtbfL^{d,\nu}(E_0,E_1)$ uses local trivializations of the 
vector-bundles and $(n_1\times n_0)$-matrices $a(x,\xi;\mu)=\big(a_{jk}(x,\xi;\mu)\big)$ where 
the symbols $a_{jk}$ are from the corresponding symbol-classes $\wt{S}^{d,\nu}_{1,0}$, etc. 
We leave the details to the reader. 

As above, given a bundle $E$, a family of order-reducing operators is a set $\Lambda_E$ of operators 
$\Lambda^\alpha_E(\mu)\in L^{\alpha}(E,E)$, $\alpha\in\rz$, which have (scalar-valued) 
principal symbol $\lambda^\alpha_0(x,\xi;\mu)=(|\xi|_x^2+\mu^2)^{\alpha/2}$ and such that 
$\Lambda^0_E(\mu)$ is the identity operator. Then we obtain: 

\begin{theorem}\label{thm:expansion-bundles}
Let $A(\mu)\in\wtbfL^{d,\nu}_{1,0}(E_0,E_1)$. Then there exists uniquely determined operators 
$A^\infty_{[\nu+j]}\in L^\nu_{1,0}(E_0,E_1)$, $j\in\nz_0$, such that, for every $N\in\nz$, 
 $$A(\mu)=\sum_{j=1}^{N-1}A^\infty_{[\nu+j]}\Lambda_{E_0}^{d-\nu-j}(\mu)
    \mod \wt{L}^{d,\nu+N}_{1,0}(E_0,E_1).$$
The leading coefficient $A^\infty_{[\nu]}$ is called the limit-operator of $A(\mu)$; 
it behaves multiplicatively under composition. 
\end{theorem}

\subsection{Symbolic structure and ellipticity in $\wtbfL^{d,\nu}(E_0,E_1)$}
\label{sec:07.4}

With any $A(\mu)\in\wtbfL^{d,\nu}(E_0,E_1)$ we associate:   
\begin{itemize}
 \item[$(1)$] the homogeneous principal symbol 
   $$\sigma(A)\in \wtbfS^{d,\nu}_{hom}((T^*M\setminus0)\times\rpbar;E_0,E_1)$$ 
  $($a homomorphism acting between the pull-backs to $(T^*M\setminus0)\times\rpbar$ 
  of the bundles $E_0$ and $E_1$, respectively), 
 \item[$(2)$] the principal angular symbol 
   $$\wh{\sigma}(A)\in S^{d,\nu}_{hom}(T^*M\setminus0;E_0,E_1)$$
  $($a homomorphism acting between the pull-backs to $T^*M\setminus0$ of the bundles 
  $E_0$ and $E_1$, respectively).  
\item[$(3)$] the principal limit-operator $A^\infty_{[\nu]}\in L^\nu(M;E_0,E_1)$.  
\end{itemize}
Recall the compatibility relation 
\begin{equation}\label{eq:compatible}
 \wh{\sigma}(A)=\sigma(A^\infty_{[\nu]}), 
\end{equation}
i.e., the principal angular symbol coincides with the homogeneous principal symbol 
of the limit-operator.

\begin{proposition}\label{prop:rough-parametrix}
Let $A(\mu)\in\wtbfL^{d,\nu}(E_0,E_1)$ and assume that both 
homogeneous principal symbol and principal angular symbol are
invertible on their domains. Then there exists a $($rough$)$ parametrix
$B(\mu)\in\wtbfL^{-d,-\nu}(E_1,E_0)$, i.e., 
 $$A(\mu)B(\mu)-1\in \wtbfL^{0-\infty,0-\infty}(E_1,E_1),\qquad 
     B(\mu)A(\mu)-1\in\wtbfL^{0-\infty,0-\infty}(E_0,E_0).$$  
\end{proposition}

This result follows from the fact that the invertibility of a homogeneous principal symbol 
belonging to $\wtbfS^{d,\nu}_{hom}((T^*M\setminus0)\times\rpbar;E_0,E_1)$ together with the 
invertibility of its angular symbol implies that its inverse belongs to the class 
$\wtbfS^{-d,-\nu}_{hom}((T^*M\setminus0)\times\rpbar;E_1,E_0)$, 
cf. the local situation mentioned after Definition \ref{def:taylor}. 

\begin{definition}\label{def:ellipticity}
We call $A(\mu)\in\wtbfL^{d,\nu}(E_0,E_1)$ elliptic if its homogeneous principal symbol is invertible 
on its domain and its principal limit-operator is 
invertible as a map $H^s(M,E_0)\to H^{s-\nu}(M,E_1)$ for some 
$s$\footnote{or, equivalently, there exists a $\psi$do $L^{-\nu}_{1,0}(M;E_1,E_0)$ which is the 
inverse.}.
\end{definition}

\begin{theorem}\label{thm:parametrix}
Let $A(\mu)\in\wtbfL^{d,\nu}(E_0,E_1)$ be elliptic. 
Then there exists a parametrix
$B(\mu)\in\wtbfL^{-d,-\nu}(E_1,E_0)$ and a $\mu_0\ge0$ such that 
 $$A(\mu)^{-1}=B(\mu),\qquad\mu\ge\mu_0.$$
\end{theorem}
\begin{proof}
By Proposition \ref{prop:rough-parametrix}, there exists a rough parametric 
$B_0(\mu)\in\wtbfL^{-d,-\nu}(E_1,E_0)$ such that  
$A(\mu)B_0(\mu)=1-R_0(\mu)$ with $R_0(\mu)\in \wtbfL^{0-\infty,0-\infty}(E_1,E_1)$. 

Now let $B_1(\mu):=B_0(\mu)+ (A^\infty_{[\nu]})^{-1}R^\infty_{0,[-\infty]}\Lambda^{-(d-\nu)}$. Then 
 $$B_1(\mu)-B_0(\mu)\in \wtbfL^{-d-\infty,-\nu-\infty}(E_1,E_0);$$ 
hence $B_1(\mu)$ is also a rough parametrix of $A(\mu)$, i.e.,  
$A(\mu)B_1(\mu)=1-R_1(\mu)$ with $R_1(\mu)\in \wtbfL^{0-\infty,0-\infty}(E_1,E_1)$. 
Moreover, $R_1(\mu)$ has vanishing limit-operator, since  
 $$R^\infty_{1,[-\infty]}
    =1-A^\infty_{[\nu]}\big(B^\infty_{0,[-\nu]}+(A^\infty_{[\nu]})^{-1}R^\infty_{0,[-\infty]}\big)
    =1-(1-R^\infty_{0,[-\infty]})-R^\infty_{0,[-\infty]}=0.$$
Then, arguing as in Proposition \ref{prop:1+r}, there exists an 
$S_1(\mu)\in \wtbfL^{0-\infty,0-\infty}(E_1,E_1)$ with vanishing limit-operator such that 
$(1-R_1(\mu))(1-S_1(\mu))=0$ for sufficiently large $\mu$. Thus, $B(\mu):=B_1(\mu)(1-S_1(\mu))$ 
is a parametrix which yields a right-inverse of $A(\mu)$ for large $\mu$. 
Since we can construct in the same way a left-inverse of $A(\mu)$ for large $\mu$, the claim follows. 
\end{proof}

\subsection{Operators with finite regularity number}
\label{sec:07.5}

In analogy to Section \ref{sec:05.4} we introduce the class 
 $$\mathbf{L}^{d,\nu}(E_0,E_1)=\wtbfL^{d,\nu}(E_0,E_1)+L^{d}(E_0,E_1),\qquad \nu\in\gz.$$
If $\nu$ is positive, the homogeneous principal symbol extends to a bundle homomorphism on 
$(T^*M\times\rpbar)\setminus 0$ and ellipticity means invertibility of this extended symbol. 
Then Theorem \ref{thm:param02} generalizes in the obvious way to the global setting. 

\subsection{Resolvent trace expansion}
\label{sec:07.6}

Let us return to the resolvent-trace expansion of Grubb-Seeley. 
Let $\Lambda=\{re^{i\theta}\mid 0\le\theta\le\Theta\}$ and let $A\in L^m(M;E,E)$, $m\in\nz$, 
be a $\psi$do such that $\lambda-\sigma(A)$ is invertible on 
$(T^*M\times\Lambda)\setminus 0$.  Moreover, let $Q\in L^\omega(M;E,E)$. 
Theorem \ref{thm:trace-expansion-main} together with integration over $M$ yield the following: 

\begin{theorem}\label{thm:trace-expansion-main-02}
With the above notation and assumptions and $\ell\in\nz$ such that 
$\omega-m\ell<-\mathrm{dim}\,M$, 
there exist numbers $c_j,c_j^\prime,c_j^{\prime\prime}$, $j\in\nz_0$, such that 
\begin{equation}\label{exp}
 \mathrm{Tr}\, Q(\lambda-A)^{-\ell}\sim \sum_{j=0}^{+\infty}c_j
   \lambda^{\frac{n+\omega-j}{m}-\ell}
   +\sum_{j=0}^{+\infty}
    \big(c_j^\prime\log\lambda+c_j^{\prime\prime}\big)\lambda^{-\ell-\frac{j}{m}},
\end{equation}
uniformly for $\lambda\in \Lambda$ with $|\lambda|\lra+\infty$. 
Moreover, $c_j^\prime=c_j^{\prime\prime}=0$ whenever $j$ is not an integer 
multiple of $m$. 
\end{theorem}

\subsection{Pseudodifferential operators of Toeplitz type}
\label{sec:07.7}

Let us conclude with an application to so-called $\psi$do of Toeplitz type, cf. \cite{Seil12,Seil15}. 

To this end, for $j=0,1$, let $E_j$ be a vector-bundle over $M$ and $P_j\in L^0(M;E_j,E_j)$ 
be idempotent, i.e., $P_j^2=P_j$. The $P_j$ define closed subspaces 
 $$H^s(M,E_j;P_j):=P_j\big(H^s(M,E_j)\big)\subseteq H^s(M,E_j)$$
in the scale of $L_2$-Sobolev spaces $H^s$. 
Given $A(\mu)\in\wtbfL^{d,0}(E_0,E_1)$, consider 
\begin{equation}\label{eq:toeplitz01}
 \bfA(\mu):= P_1A(\mu)P_0\in \wtbfL^{d,0}(E_0,E_1). 
\end{equation} 
We are interested in the invertibilty of 
\begin{equation}\label{eq:toeplitz02}
 \bfA(\mu)\colon H^{s}(M,E_0;P_0)\lra H^{s-d}(M,E_1;P_1).
\end{equation} 
Consider $P_j$ as an element in $\wtbfL^{0,0}(E_j,E_j)$. 
Since $P_j$ is idempotent, so is the homogeneous principal symbol $\sigma(P_j)$ as 
morphism of the pull-back of $E_j$ to $(T^*M\setminus0)\times\rpbar$, hence defines a 
sub-bundle denoted by $E_j(P_j)$.   

\begin{theorem}\label{thm:toeplitz01}
Let notations be as above. Assume that 
\begin{itemize}
 \item[i$)$] $\sigma(\mathbf{A}):E_0(P_0)\to E_1(P_1)$ is invertible, 
 \item[ii$)$] $P_1A^\infty_{[0]}P_0:H^s(M,E_0;P_0)\to H^{s}(M,E_1;P_1)$ 
  is invertible for some $s$.
\end{itemize} 
Then there exists a $B(\mu)\in\wtbfL^{-d,0}(E_1,E_0)$ such that, for 
$\mathbf{B}(\mu):= P_0B(\mu)P_1$,  
\begin{align}\label{eq:toep01}
 \mathbf{B}(\mu)\mathbf{A}(\mu)=P_0,\qquad \mathbf{A}(\mu)\mathbf{B}(\mu)=P_1
\end{align}
for sufficiently large values of $\mu$. In particular, the map 
\eqref{eq:toeplitz02} is an isomorphism for every choice of $s$ and $\mu$ large. 
\end{theorem}
\begin{proof}
For $j=0,1$ let $S_j(\mu)\in L^{jd-s}(E_j,E_j)$ be invertible for every $\mu\ge0$ with 
$S_j(\mu)^{-1}\in L^{s-dj}(E_j,E_j)$. Then 
\begin{align*} 
 P_j^\prime(\mu)&:=S_j(\mu)^{-1}P_jS_j(\mu)\in L^0(E_j,E_j),\\ 
 A^\prime(\mu)&:=S_1(\mu)^{-1}A(\mu)S_0(\mu)\in \wtbfL^{0,0}(E_0,E_1).
\end{align*}  
Note that the $P_j^\prime(\mu)$ are $($parameter-dependent$)$ idempotents. 

If $\bfA^\prime(\mu)=P_1^\prime(\mu) A^\prime(\mu)P_0^\prime(\mu)$ has a parametrix 
$\mathbf{B}^\prime(\mu)= P_0^\prime(\mu) B^\prime(\mu)P_1^\prime(\mu)$ with 
$B^\prime(\mu)\in\wtbfL^{0,0}(E_1,E_0)$ $($i.e., the analog of \eqref{eq:toep01} is true$)$, 
then $B(\mu):=S_0(\mu)B^\prime(\mu)S^{-1}_1(\mu)$ yields the desired parametrix 
$\mathbf{B}(\mu)$. 
However, this result follows from the general theory of abstract pseusodifferential operators 
and associated Toeplitz operators developed in \cite{Seil12,Seil15}. 
In fact, in the notation of \cite[Section 3.1]{Seil15} let  $\Lambda=\rpbar$, let 
 $$G=\{g=(M,E)\mid \text{$E$ vector-bundle over $M$}\}$$
be the set of all admissible weights and let $H^0(g)=L^2(M,E)$ for $g=(M,E)$. 
Moreover, for $g_0=(M,E_0)$, $g_1=(M,E_1)$, and $\mathfrak{g}=(g_0,g_1)$ let 
$L^0(\mathfrak{g})=\wtbfL^{0,0}(M;E_0,E_1)$ and  
 $$L^{-\infty}(\mathfrak{g})
    =\big\{A(\mu)\in \wtbfL^{0-\infty,0-\infty}(M;E_0,E_1)\mid A^\infty_{[0]}=0\big\}.$$
Now we can apply in \cite[Theorem 1, Section 3.2]{Seil15}, 
noting that i), ii) give the required hypotheses. 
\end{proof}

As a particular case we can take $\mathbf{A}(\mu)=P_1(\mu^d-A)P_0$ with a $\psi$do 
$A\in L^d(M;E,E)$, $d\in\nz$, and two idempotents $P_0,P_1\in L^0(M;E,E)$. Note that 
$A(\mu)=\mu^d-A$ considered as an element of $\wtbfL^{d,0}$ has limit-operator 
$A_{[0]}\equiv 1$, hence condition ii) in Theorem \ref{thm:toeplitz01} reduces to the 
requirement that $P_1:H^s(M,E;P_0)\to H^{s}(M,E;P_1)$ isomorphically for some $s$. 

\section{Appendix: Calculus for symbols with expansion at infinity} 

Let us provide the detailed proofs of Theorems \ref{thm:Leibniz product} and 
\ref{thm:adjoint}. They are based on the concept of oscillatory integrals in the spirit 
of \cite{Kuma}, but extended to Fr\`{e}chet space valued amplitude functions. 
For an account on this concept see \cite{GSS}.

Let $E$ be a Fr\'echet space whose topology is described by a system of semi-norms 
$p_n$, $n\in\nz$. A smooth function $q=q(y,\eta):\rz^m\times\rz^m\to E$ is called an 
amplitude function with values in $E$, 
provided there exist sequences $(m_n)$ and $(\tau_n)$ such that 
 $$p_n\big(D^\gamma_\eta D^\delta_y q(y,\eta)\big)\lesssim \spk{y}^{\tau_n}\spk{\eta}^{m_n}$$
for all $n$ and for all orders of derivatives. The space of such amplitude functions is denoted by 
$\calA(\rz^m\times\rz^m,E)$. We shall frequently make use of the following simple observation$:$

\begin{lemma}
Let $E_0,E_1$ and $E$ be Fr\'echet spaces and let $(\!(\cdot,\cdot)\!)$ be a bilinear continuous 
map from $E_1\times E_0$ to $E$. If $q_j(y,\eta)$ are amplitude functions with values in 
$E_j$, $j=0,1$, then $q(y,\eta):=(\!(q_1(y,\eta),q_0(y,\eta))\!)$ is an amplitude function 
with values in $E$. 
\end{lemma}

If $\chi(y,\eta)$ denotes a cut-off function with $\chi(0,0)=1$, the so-called oscillatory integral 
 $$\mathrm{Os}-\iint e^{-iy\eta}q(y,\eta)\,dy\dbar\eta:=
   \lim_{\eps\to0} \iint_{\rz^n\times\rz^n} e^{-iy\eta}\chi(\eps y,\eps\eta)q(y,\eta)\,dy\dbar\eta$$
exists and is independent of the choice of $\chi$. Note that for a continuous, $E$-valued function 
$f$ with compact support, $\iint f(y,\eta)\,dy d\eta$ is the unique element $e\in E$ such that 
$\spk{e^\prime,e}=\iint \spk{e^\prime,f(y,\eta)}\,dy d\eta$ for every functional 
$e^\prime\in E^\prime$. For simplicity of notation we shall simply write $\iint$ rather 
than $\mathrm{Os}-\iint$. 

\begin{proposition}\label{prop:amplitude}
Let $a\in\wtbfS^{d,\nu}_{1,0}$. Then 
 $$q(y,\eta):=\Big( (x,\xi,\mu)\mapsto a(x+y,\xi+\eta;\mu)\Big),\qquad y,\eta\in\rz^n,$$
defines an amplitude function $q\in \calA(\rz^n\times\rz^n,\wtbfS^{d,\nu}_{1,0})$. 
The principal-limit symbol is  
 $$q_{[\nu]}(y,\eta)=
    \Big( (x,\xi)\mapsto a_{[\nu]}^\infty(x+y,\xi+\eta)\Big),\qquad y,\eta\in\rz^n,$$
$($it is an amplitude function with values in $S^\nu_{1,0}(\rz^n))$. 

Analogous results hold true for $q_1(\eta):=((x,\xi,\mu)\mapsto a(x,\xi+\eta;\mu))$ and 
$q_2(y):=((x,\xi,\mu)\mapsto a(x+y,\xi;\mu))$. 
\end{proposition}
\begin{proof}
\textbf{Step 1:} Suppose first that $a\in\wt{S}^{d,\nu}_{1,0}$ only. We show that $q$ is an 
amplitude function with values in $\wt{S}^{d,\nu}_{1,0}$. 

Recall that the topology of $\wt{S}^{d,\nu}_{1,0}$ is defined by the semi-norms 
 $$p_N(a):=\max_{|\alpha|+|\beta|+j\le N}\sup_{x,\xi,\mu}
   |D^\alpha_\xi D^\beta _x D^j_\mu a(x,\xi;\mu)|\spk{\xi}^{|\alpha|-\nu}\spk{\xi,\mu}^{\nu-d+j}.$$
If $|\alpha|+|\beta|+j\le N$ and $\gamma,\delta\in\nz_0^n$ are arbitrary, then 
\begin{align*}
|D^\alpha_\xi D^\beta_x& D^j_\mu D^\gamma_\eta D^\delta_y {a}(x+y,\xi+\eta;\mu)|
\le C_{\gamma,\delta,N}\spk{\xi+\eta,\mu}^{d-\nu-j}\spk{\xi+\eta}^{\nu-|\alpha|-|\gamma|}\\
&\le C_{\gamma,\delta,N}\spk{\xi,\mu}^{d-\nu-j}\spk{\eta}^{|d-\nu-j|}
\spk{\xi}^{\nu-|\alpha|}\spk{\eta}^{|\nu-|\alpha||}\\
&\le C_{\gamma,\delta,N}\spk{\xi,\mu}^{d-\nu-j}\spk{\xi}^{\nu-|\alpha|}\spk{\eta}^{m_N}
\end{align*}
with $m_N=\max\{|d-\nu-j|+|\nu-|\alpha||\mid |\alpha|+j\le N\}$. This shows  
 $$p_N\big(D^\gamma_\eta D^\delta_y q(y,\eta)\big)\lesssim \spk{\eta}^{m_N}.$$

\textbf{Step 2:} Suppose $a\in\wtbfS^{d,\nu}_{1,0}$. Then 
\begin{align*} 
 a(x+y,\xi+\eta;\mu)=&\sum_{j=0}^{N-1}a_{[\nu+j]}^\infty(x+y,\xi+\eta)[\xi+\eta,\mu]^{d-\nu-j}+\\
   &+r_{a,N}(x+y,\xi+\eta;\mu).
\end{align*}
According to Step 1, $r_{a,N}(x+y,\xi+\eta;\mu)$ defines an amplitude function with values in 
$\wt{S}^{d,\nu+N}_{1,0}$. In the same way one sees that $a_{[\nu+j]}^\infty(x+y,\xi+\eta)$ 
defines an amplitude function with values in ${S}^{\nu+j}_{1,0}(\rz^n)$. 
By the following Lemma \ref{lem:amplitude}, $[\xi+\eta,\mu]^{d-\nu-j}$ defines an amplitude 
with values in $S^{d-\nu-j}$ hence, due to Proposition \ref{prop:classical}, 
with values in $\wtbfS^{d-\nu-j,0}_{1,0}$. Thus we can write, for every $M$,  
 $$[\xi+\eta;\mu]^{d-\nu-j}=\sum_{\ell=0}^{M-1}
   p_{j,[d-\nu-j+\ell]}^\infty(\eta,\xi)[\xi,\mu]^{d-\nu-j-\ell}+r_{j,M}(\eta,\xi;\mu),$$
where $p_{j,[d-\nu-j+\ell]}^\infty(\eta,\xi)$ defines an amplitude function with values in 
$S^{d-\nu-j+\ell}_{1,0}(\rz^n)$ and $r_{j,M}(\eta,\xi;\mu)$ an amplitude function 
with values in $\wt{S}^{d-\nu-j,M}$. Note that $p_{j,[d-\nu-j]}^\infty(\eta,\xi)\equiv 1$. 
Inserting these expansions above and re-arranging terms, we find an expansion 
 $$a(x+y,\xi+\eta;\mu)=\sum_{j=0}^{N-1}\wt{a}_{[\nu+j]}^\infty(y,\eta;x,\xi)[\xi,\mu]^{d-\nu-j}
   +R_{a,N}(y,\eta;x,\xi,\mu)$$
where $\wt{a}_{[\nu+j]}^\infty(y,\eta;x,\xi)$ and $R_{a,N}(y,\eta;x,\xi,\mu)$ define 
amplitude functions with values in $\wt{S}^{d-\nu+j}_{1,0}(\rz^n)$ and 
$\wt{S}_{1,0}^{d-\nu,N}$, respectively. 
Note that $\wt{a}_{[\nu]}^\infty(y,\eta;x,\xi)=a_{[\nu]}^\infty(x+y,\xi+\eta)$. 
Altogether, this shows the claims for $q$. 

$q_1$ and $q_2$ are handled in the same way. 
\end{proof}

\begin{lemma}\label{lem:amplitude}
Let $a(\xi;\mu)\in S^d$. Then 
 $$q(\eta)=\Big( (\xi,\mu)\mapsto a(\xi+\eta;\mu)\Big),\qquad \eta\in\rz^n,$$ 
defines an amplitude function with values in $S^d$.  
\end{lemma}
\begin{proof}
By Taylor expansion, 
$$a(\xi+\eta;\mu)=\sum_{|\alpha|=0}^{N-1}
   \frac{1}{\alpha!}\partial^\alpha_\xi a(\xi;\mu)\eta^\alpha+r_N(\eta,\xi;\mu)$$
with 
 $$r_N(\eta,\xi;\mu)=N\sum_{|\sigma|=N}\frac{\eta^\sigma}{\sigma!}
   \int_0^1(1-\theta)^{N-1}(\partial^\sigma_\xi a)(\xi+\theta\eta;\mu)\,d\theta.$$
Denoting by $r_{N,\sigma}(\eta,\xi;\mu)$ the integral term, we have 
\begin{align*}
 |D^\gamma_\eta D^\alpha_\xi D^j_\mu r_{N,\sigma}(\eta,\xi;\mu)|
 &\lesssim \int_0^1\spk{\xi+\theta\eta,\mu}^{d-N-|\alpha|-|\gamma|}\,d\theta
 \lesssim \spk{\xi,\mu}^{d-N-|\alpha|}\spk{\eta}^{|d-N-|\alpha||}.
\end{align*} 
We conclude that $r_N(\eta,\xi;\mu)$ defines an amplitude function with values in $S^{d-N}_{1,0}$.  
Write  
 $$\partial^\alpha_\xi a(\xi;\mu)
   =\chi(\xi,\mu)\sum_{j=0}^{N-1-|\alpha|}\partial^\alpha_\xi a_{j}(\xi;\mu)+
    s_{\alpha,N}(\xi;\mu)$$
with a zero-excision function $\chi$, $a_j\in S^{d-j}_{hom}$, and $s_{\alpha,N}\in S^{d-N}_{1,0}$. 
We find that 
$$a(\xi+\eta;\mu)=\chi(\xi,\mu)\sum_{|\alpha|+j=0}^{N-1}
   \frac{1}{\alpha!}\partial^\alpha_\xi a_{j}(\xi;\mu)\eta^\alpha+R_N(\eta,\xi;\mu)$$
with an amplitude function $R_N(\eta,\xi;\mu)$ taking values in $S^{d-N}_{1,0}$. Therefore 
 $$[q(\eta)_{k}](\xi;\mu)=\sum_{|\alpha|+j=k}\frac{1}{\alpha!}\partial^\alpha_\xi  
   a_{j}(\xi;\mu)\eta^\alpha,\qquad k\in\nz_0,$$
which are obviously amplitude functions with values in $S^{d-k}_{hom}$. 
According to the definition of the topology of $S^d$ this shows the claim.  
\end{proof}

\begin{theorem}
The Leibniz product induces continuous maps 
 $$(a_1,a_0)\mapsto a_1\#a_0:\wtbfS^{d_1,\nu_1}_{1,0}\times \wtbfS^{d_0,\nu_0}_{1,0}\lra 
     \wtbfS^{d_0+d_1,\nu_0+\nu_1}_{1,0}$$
and the limit-symbol behaves multiplicatively$:$ 
$(a_1\#a_0)_{[\nu_0+\nu_1]}^\infty=a_{1,[\nu_1]}^{\infty}\# a_{0,[\nu_0]}^{\infty}$. 
Moreover, for every $N\in\nz_0$, 
\begin{equation}\label{eq:leibniz-expansion}
  a_1\#a_0\equiv\sum_{|\alpha|=0}^{N-1}\frac{1}{\alpha!}(\partial^\alpha_\xi a_1)(D^\alpha_x a_0)
    \mod \wtbfS^{d_0+d_1-N,\nu_0+\nu_1-N}_{1,0}.
\end{equation}
\end{theorem}
\begin{proof}
Recall that 
 $$(a_1\#a_0)(x,\xi;\mu)=\iint e^{-iy\eta}\underbrace{a_1(x,\xi+\eta;\mu)a_0(x+y,\xi;\mu)}_{
     =:p(y,\eta;x,\xi,\lambda)}\,dy\dbar\eta.$$
By Proposition \ref{prop:amplitude}, $p$ is an amplitude function with values in 
$\wtbfS^{d_0+d_1,\nu_0+\nu_1}_{1,0}$, hence the oscillatory integral converges 
in this space. 

Since the map $a\mapsto a_{[\nu]}^\infty:\wtbfS^{d,\nu}_{1,0}\to S^{\nu}_{1,0}$ 
is linear and continuous, we find that 
\begin{align*}
 (a_1\#a_0)_{[\nu_0+\nu_1]}^\infty(x,\xi)
 =&\iint e^{-iy\eta}p(y,\eta;x,\xi;\mu)_{[\nu_0+\nu_1]}^\infty\,dy\dbar\eta\\
 =&\iint e^{-iy\eta}a_{1,[\nu_1]}^{\infty}(x,\xi+\eta)
   a_{0,[\nu_0]}^{\infty}(x+y,\xi)\,dy\dbar\eta\\
 =&(a_{0,[\nu_0]}^{\infty}\#a_{1,[\nu_1]}^{\infty})(x,\xi).
\end{align*}
Concerning the expansion \eqref{eq:leibniz-expansion} recall that the difference of $a_1\#a_0$ 
and the sum in \eqref{eq:leibniz-expansion} is given by 
\begin{align}\label{eq:leibniz-remainder}
\begin{split} 
 r_N(x,\xi;\mu)=&N\sum_{|\sigma|=N}\int_0^1\frac{(1-\theta)^{N-1}}{\sigma!}\times\\
 &\times\iint e^{-iy\eta}\partial^{\sigma}_\xi a_1(x,\xi+\theta\eta;\mu)D^{\sigma}_x 
 a_0(y+x,\xi;\mu)\,dy\dbar\eta\,d\theta.
\end{split}
\end{align}
Similarly as before, one can show that the integrand in \eqref{eq:leibniz-remainder} 
is an amplitude function with values in $\wtbfS^{d_0+d_1-N,\nu_0+\nu_1-N}_{1,0}$, 
depending continuously on $\theta$. This yields the claim. 
\end{proof}

The proof of Theorem \ref{thm:adjoint} is analogous, using  
 $$a^{(*)}(x,\xi;\mu)=\iint e^{-iy\eta}\overline{a(x+y,\xi+\eta;\mu)}\,dy\dbar\eta,$$
and a similar formula for the remainder in \eqref{eq:leibniz-expansion2}.  	

\bibliographystyle{amsalpha}

\end{document}